\documentclass{article}

\usepackage[hang]{subfigure}
\usepackage{graphicx}
\usepackage{wrapfig}
\usepackage{url}
\usepackage{amsmath, amsthm, amssymb}
\usepackage{listings}
\usepackage{subfig}
\usepackage{verbatim}
\usepackage{xspace}
\usepackage[top=1.0in, bottom=1.0in, left=1.0in, right=1.0in]{geometry}
\usepackage[usenames,dvipsnames]{color}
\usepackage{authblk}
\usepackage{pdfpages}
\usepackage{mathptmx}      

\newtheorem{theorem}{Theorem}
\newtheorem{lem}[theorem]{Lemma}
\newtheorem{prop}[theorem]{Proposition}

\newtheorem{definition}[theorem]{Definition}

\newtheorem{examp}{Example}

\def\R{\mathbb{R}}
\def\V{\mathsf{V}}
\def\E{\mathsf{E}}
\def\F{\mathsf{F}}

\title{Liftings and stresses for planar periodic frameworks}
\author{Ciprian S. Borcea and Ileana Streinu}

\date{}

\begin{document}
\maketitle

\begin{abstract}
We formulate and prove a periodic analog of Maxwell's theorem relating stressed planar frameworks and their liftings to polyhedral surfaces with spherical topology. We use our lifting theorem to prove deformation and rigidity-theoretic properties for planar periodic pseudo-triangulations, generalizing features known for their finite counterparts. These properties are then applied to questions originating in mathematical crystallography and materials science, concerning planar periodic auxetic structures and ultrarigid periodic frameworks.


\end{abstract}

\section{Introduction}
\label{sec:intro}

A remarkable correspondence between planar stressed graphs, their duals and polyhedral surfaces with a spherical topology has been established in 1870 by James Clerk Maxwell: 

\begin{wrapfigure}{l}{0.34\textwidth}
\vspace{-12pt}
\centering
{\includegraphics[width=0.34\textwidth]{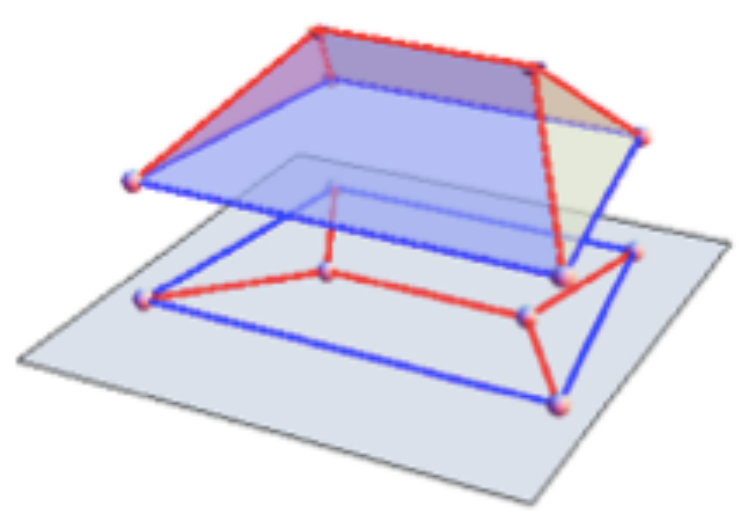}}
\vspace{-14pt}
\caption{A finite planar stressed graph and a Maxwell lifting.}
\vspace{-20pt}
 \label{fig:liftFiniteGraph}
\end{wrapfigure}

\medskip
\noindent
{\bf Maxwell's Theorem \cite{M2} }
{\em A planar geometric graph $(G,p)$ supports a non-trivial stress on its edges if and only if it has a dual reciprocal diagram and, at the same time, if and only if it has a non-trivial lifting to 3D as a polyhedral terrain.}

\medskip
\noindent
The necessary definitions are recalled below in Section \ref{sec:preliminaries}. A closely related instance of this theorem is the classical duality between Voronoi diagrams and Delaunay tesselations, where the 3D lifting is onto a paraboloid.  Maxwell's diagrams, further popularized in Cremona's book \cite{Cremona}, were widely used for engineering calculations throughout the 19th and 20th centuries. The theorem has many other applications, for example in problems of robustness for geometric algorithms, rigidity theory, polyhedral combinatorics and computational geometry
\cite{CW,hopcroft:kahn:paradigmRobustAlg:1992,RG,rote:ribo:schulz:smallGridEmbeddings:2011,CDR,S2}. Most relevant to our undertaking is its role in establishing the existence of planar expansive motions used in the solution to the Carpenter's Rule problem \cite{CDR}, and in proving the expansive properties of pointed pseudo-triangulation mechanisms that are central to the algorithm for convexifying simple planar polygons of \cite{S1,S2}.

\paragraph{\bf Our results.}
In this paper we prove the following {\em periodic} analog of Maxwell's theorem. Figure~\ref{fig:stressedPeriodicFmk} illustrates the concepts, whose precise definitions will be given in Section \ref{sec:periodicLiftingStress}.

\medskip
\noindent{\bf Main Theorem }
{\em 
Let $(G,\Gamma,p,\pi)$ be a planar non-crossing periodic framework. A stress induced by a periodic lifting is a periodic stress and conversely, any periodic stress is induced by a periodic lifting, determined up to an arbitrary additive constant.
}

\begin{figure}[t]
\centering
\subfigure[]{\includegraphics[width=0.23\textwidth]{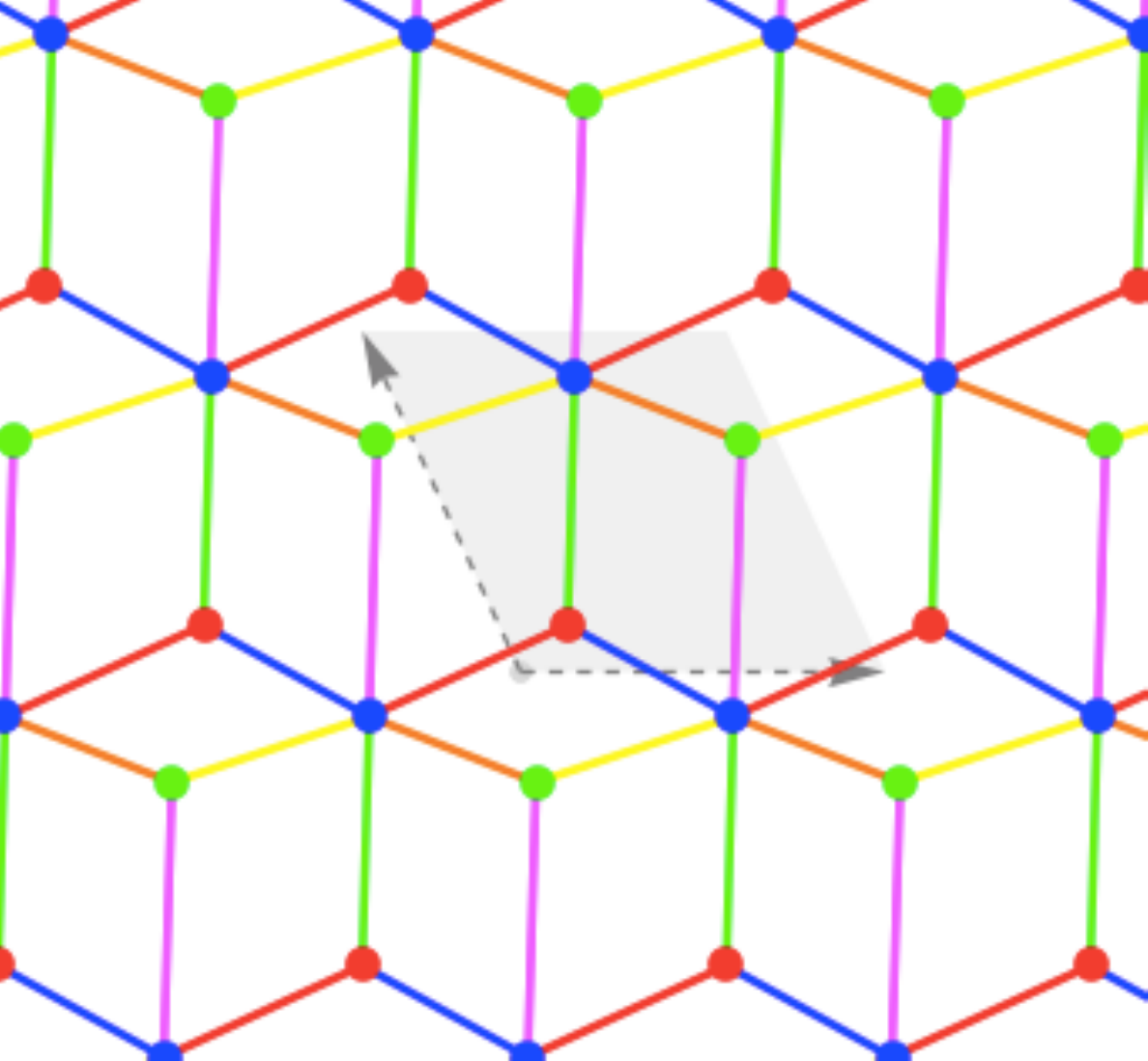}}
\hspace{3pt}
\subfigure[]{\includegraphics[width=0.23\textwidth]{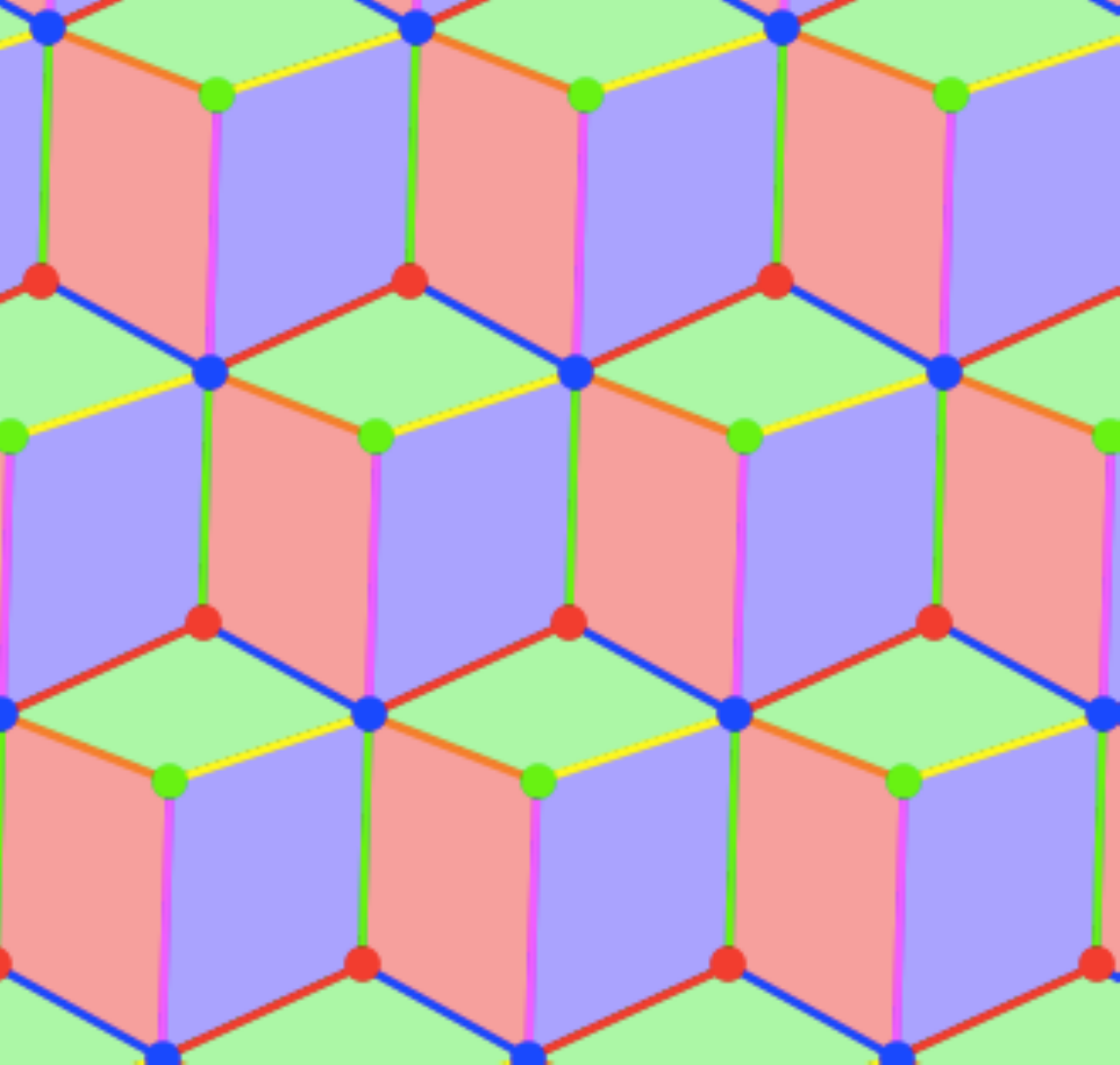}}
\hspace{20pt}
\subfigure[]{\includegraphics[width=0.22\textwidth]{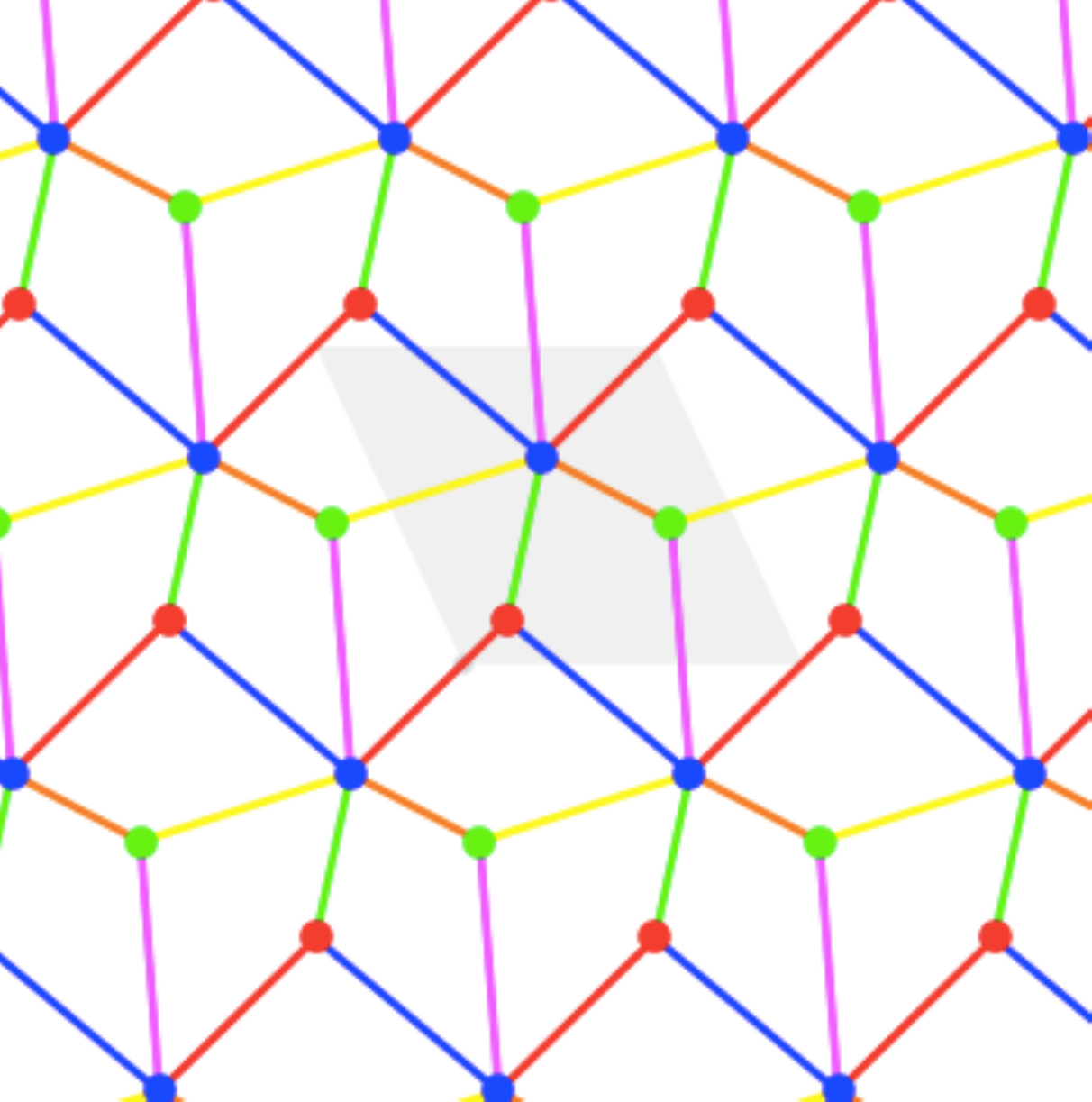}}
\hspace{3pt}
\subfigure[]{\includegraphics[width=0.22\textwidth]{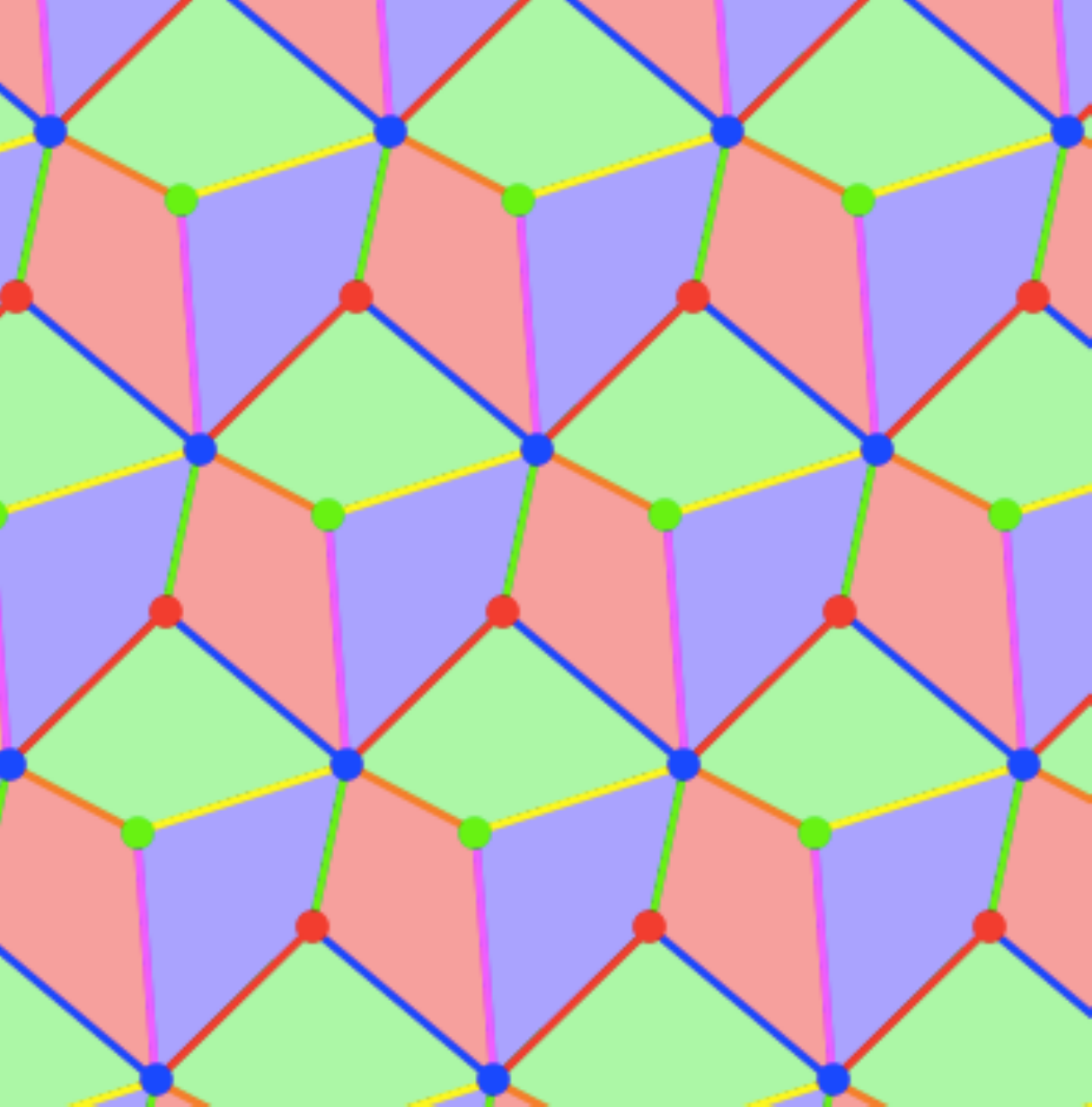}}
\vspace{-6pt}
\caption{The connection between periodic stresses and liftings proven in this paper. (a) A stressed periodic framework. Vertex and edge orbits are identically colored, and a fundamental polygon of the periodicity lattice is highlighted in grey. (b) Coloring the faces helps visualize the 3D lifting of the framework as a periodic arrangement of ``cubes''.
(c) The same periodic graph, in a placement with no non-trivial  periodic stresses. (d) The face coloring visually confirms that this is {\em not} the projection of a polyhedral surface: the faces do not look ``flat'' in 3D.}
\vspace{-10pt}
 \label{fig:stressedPeriodicFmk}
\end{figure}

\medskip
\noindent
Non-crossing periodic graphs can be seen as graphs embedded on the flat torus. However, as it will become clear from this paper, to reason on a fixed torus would be too restrictive a perspective. The most important foundational element that makes  the new result possible is our recent {\em deformation theory of periodic frameworks} \cite{BS2}, which allows the periodicity lattice to deform. The corresponding notion of {\em periodic stress} is precisely the notion of stress that is needed for the Main Theorem. This stress is more constrained than the direct generalization of the classical {\em self-stress} used for finite frameworks,  which is based solely on equilibrium at all vertices. In order to maintain the proper distinction, we refer to the latter type of stress as an {\em equilibrium stress}. 

\medskip
\noindent
Our theorem was motivated by questions arising in mathematical crystallography and computational materials science. We demonstrate its usefulness with two applications: 
constructions of ultrarigid frameworks and auxetic mechanisms.

\begin{wrapfigure}{r}{0.5\textwidth}
\vspace{-20pt}
\centering
{\includegraphics[width=0.24\textwidth]{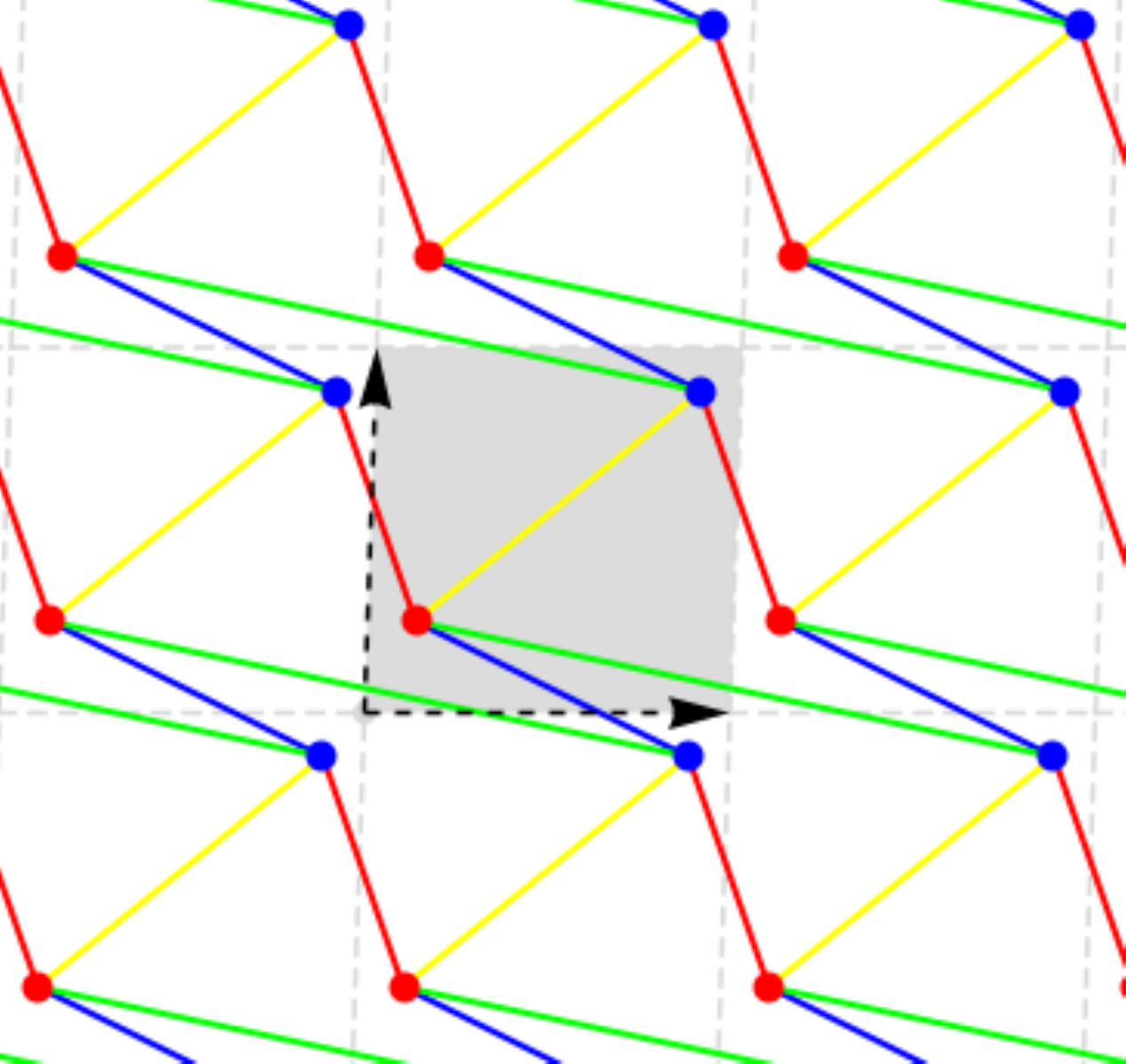}}
\hspace{4pt}
{\includegraphics[width=0.24\textwidth]{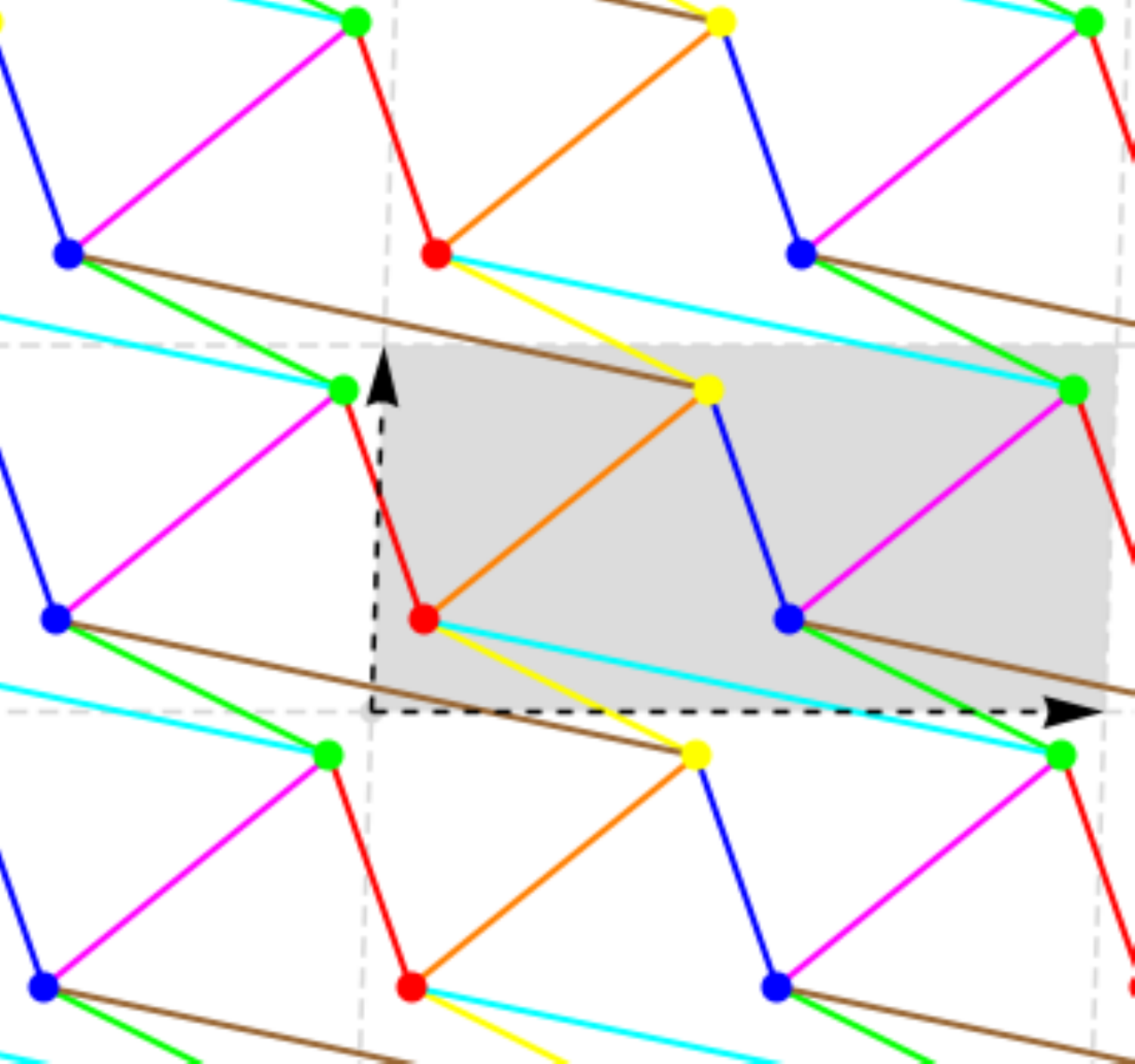}}
\caption{A periodic framework and a $2\times1$  relaxation of its periodicity lattice. 
}
\vspace{-20pt}
 \label{fig:relaxLattice}
\end{wrapfigure}

\paragraph{\bf Ultrarigidity of periodic frameworks.}  Leaving some technical details aside, our proof of the correspondence between periodic liftings and periodic stresses will proceed by showing how to obtain a transparent, algebraic matching of all the concepts involved after a sufficient {\em relaxation of periodicity}. Relaxations of periodicity are a central concern in displacive phase transitions \cite{D,BS4} and successive relaxations give rise to difficult and important problems for estimating the {\em asymptotic behavior} of a periodic framework.

\medskip
\noindent
By definition, a periodic framework is  {\em ultrarigid}  if it is and remains infinitesimally periodically rigid under arbitrary relaxations of periodicity to subgroups of finite index (see Figure~\ref{fig:relaxLattice}).  This concept was introduced in \cite{borcea:pharmacosiderite:Kavli:arxiv:2012}, and illustrated with a few examples of crystalline materials exhibiting this property. Ultrarigidity provides a rigorous tool for studying the asymptotic rigidity of a periodic framework when successive relaxations of periodicity are applied. The proof techniques developed in this paper will allow us to construct infinite families of ultrarigid examples. 

\begin{wrapfigure}{l}{0.25\textwidth}
\centering
{\includegraphics[width=0.25\textwidth]{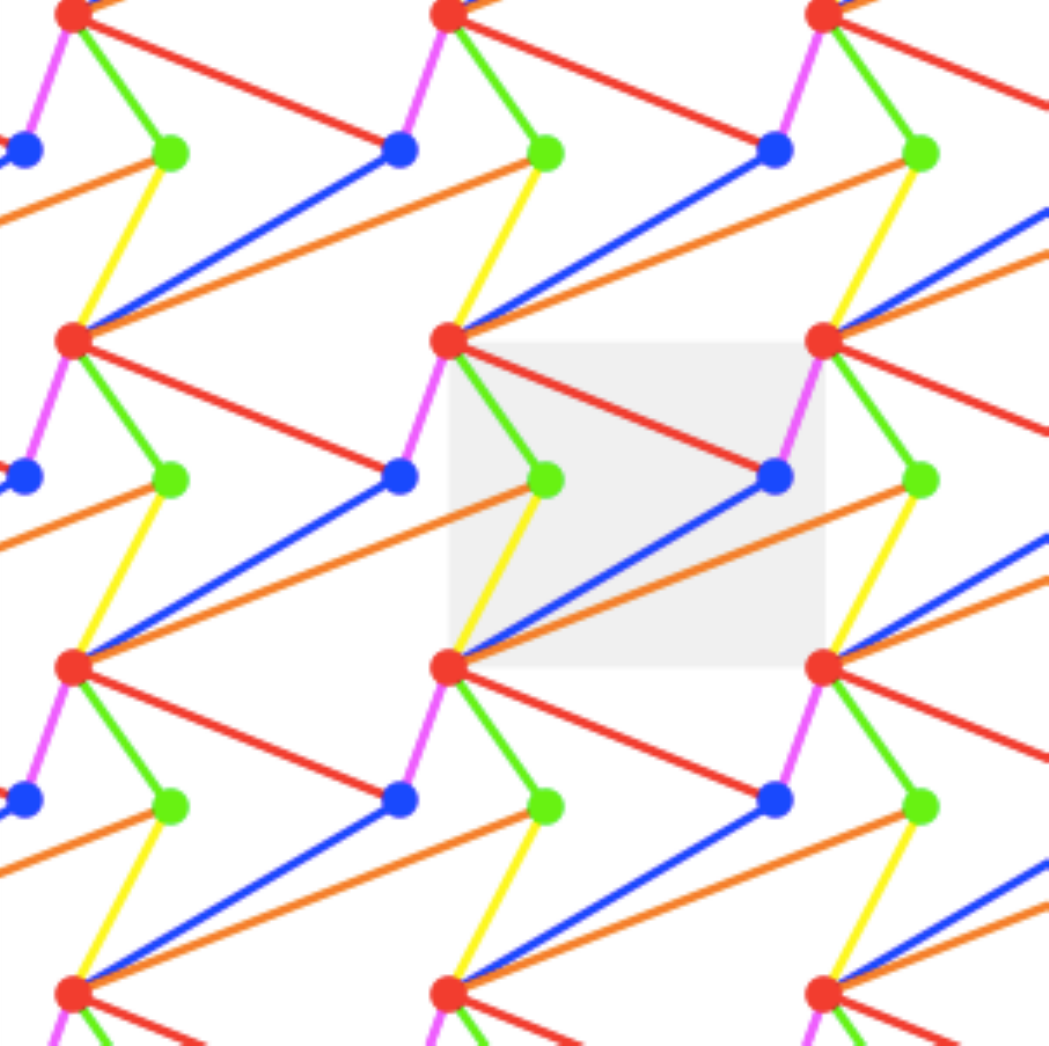}}
\vspace{-16pt}
\caption{A periodic pointed pseudo-triangulation.}
\vspace{-26pt}
\label{fig:periodicPPT}
\end{wrapfigure}

For some related considerations and observations on rigidity and relaxation, we mention \cite{Power,connellyEtAl:BallPackings:2014}. The basic theory of periodic frameworks from the point of view of rigidity and flexibility can be found in \cite{BS2,BS3,BS4}. For wider or complementary aspects of periodic framework theory we suggest \cite{Su,BST} and references therein.

\paragraph{\bf Periodic pseudo-triangulations.} We use the Main Theorem to study a new class of planar non-crossing periodic frameworks called {\em periodic pointed pseudo-triangulations} or shortly {\em periodic pseudo-triangulations}. They represent a natural analog of the finite {\em pointed} pseudo-triangulation frameworks defined and studied in  \cite{S1,S2} and possess {\em mutatis mutandis} many outstanding characteristics related to rigidity and deformations \cite{S1,S2,RSS1,RSS2}. Here we focus on the {\em expansive} one-degree-of-freedom mechanisms they provide and on the property of turning into ultrarigid periodic frameworks after  the insertion of a single (adequate) edge-orbit.

\paragraph{\bf Deformations of periodic frameworks: auxetic and expansive behavior.} 
The significance of the idea of expansive motion is well recognized in the finite setting 
\cite{CDR,S2,RSS2}: when the distance between any pair of vertices cannot decrease, self-collision of the framework is avoided. In the periodic setting, expansive mechanisms have not been explicitly considered before, although a related, yet weaker notion of {\em auxetic behavior} has recently attracted a lot of attention in materials science \cite{ENHR,LWMSE,MRMST}. Since the existing literature on auxetics is based on elasticity theory, we include a brief and necessarily selective overview of the relations existing between the purely geometric theory pursued in this paper and the larger context of  periodic structures explored in crystallography, solid state physics and materials science \cite{LdF,Su,G}.

\begin{wrapfigure}{l}{0.4\textwidth}
\centering
{\includegraphics[width=0.38\textwidth]{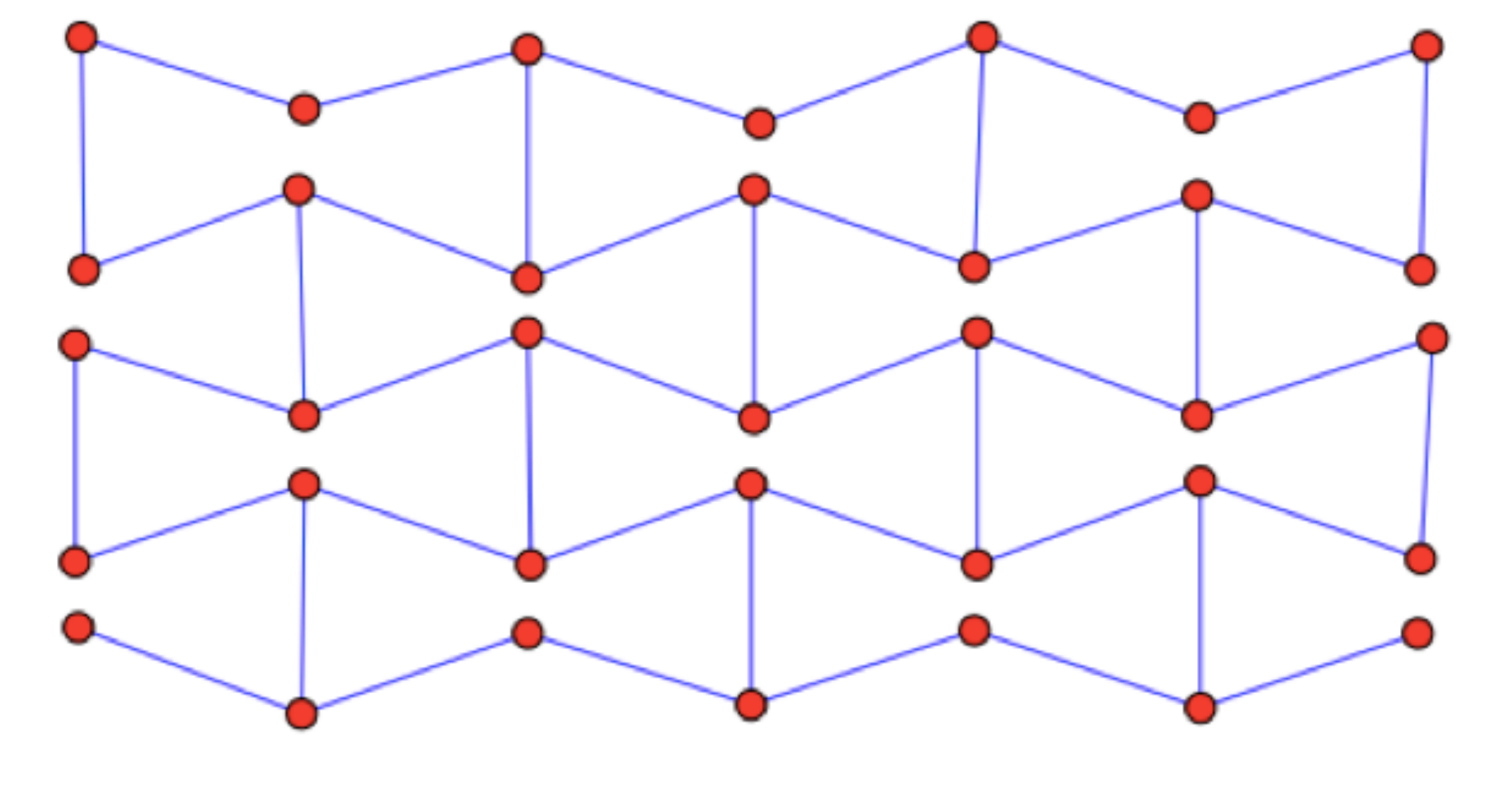}}
\vspace{-4pt}
\caption{This so-called `reentrant' structure of hexagons is often used to
illustrate auxetic behavior.}
\vspace{-10pt}
\label{fig:auxetic2D}
\end{wrapfigure}

\medskip \noindent
Interest in crystal morphology and structure motivated mathematical studies of symmetry, lattice sphere packings and crystallographic groups \cite{CS,LdF}. With the advent of X-ray diffraction, materials science gained access to atomic-scale configurations and bonding networks. An explicit mention of a {\em periodic framework deformation} appears in Pauling's 1930 paper \cite{P}. Nevertheless, such geometric investigations addressed only 
specific crystalline materials and remained mostly concerned with a number of instances of 
one-parameter deformations related to particular displacive phase transitions, as in \cite{Dol,D}. 

\medskip \noindent
The notion of {\em auxetic behavior} 
is formulated using the concept of negative Poisson's ratio \cite{G,GGLR,ENHR}, which relies on physical properties of the material: when two forces pull in opposite directions along an axis, most materials are expected to expand along this axis and to contract along directions perpendicular to it.  Auxetic behavior refers to the rather counter-intuitive lateral widening upon application of a longitudinal tensile strain. A purely geometric expression of this behavior is not anticipated in all situations. However, for periodic frameworks, we have recently proposed the general geometric notion of {\em auxetic path} in the deformation space of the periodic framework \cite{BS5}. Relying on this formulation, we prove that an expansive deformation path is necessarily an auxetic path. Periodic pseudo-triangulations thus exhibit auxetic behavior and offer an infinite supply of planar examples of ``auxetic frameworks''. By contrast, only a limited collection of sporadic and artisanal auxetic periodic examples (Figure~\ref{fig:auxetic2D}) has appeared in the literature. One may already notice that the reentrant framework from Figure~\ref{fig:auxetic2D} looks ``almost'' like a pointed pseudo-triangulation, except that it is not maximal. It needs one more edge-orbit for all its faces to become pseudo-triangles, as illustrated in Figure~\ref{fig:pptReentrant} of Section \ref{sec:auxetic}.

\paragraph{\bf Organization of the paper.} In Section \ref{sec:preliminaries} we define the basic concepts needed to study the correspondence between liftings and stresses in (finite or infinite) frameworks. Section \ref{sec:periodicLiftingStress} specializes these concepts to {\em periodic} liftings and stresses. Section \ref{sec:deformationPeriodicStress} concludes  the proof of the Main Theorem by providing the necessary link with periodic rigidity and flexibility. The Main theorem is applied in Section \ref{sec:periodicPPTs} to prove the expansive properties of periodic pseudo-triangulations. The connection with auxetic behavior and the ultrarigid character of periodic frameworks obtained from pseudo-triangulations are presented in the final Sections \ref{sec:auxetic} and \ref{sec:ultrarigidity}.

\section{Liftings and stresses}
\label{sec:preliminaries}

To formulate and prove our Main Theorem, we start with those concepts and properties that do not depend (yet) on periodicity, which is introduced in the next section.

\paragraph{\bf Graphs and frameworks.} We consider finite or countably infinite graphs which are simple (i.e. without loops or multiple edges), unoriented and of finite degree (or valency) at each vertex. Such a graph is given as a pair $G=(\V,\E)$, with $\V$ the set of vertices and $\E$ the set of (unoriented) edges. We use lower case symbols $u, v, \cdots $ for vertices in $\V$. An edge $e=\{u,v\}=\{v,u\}\in \E$ has two {\em endpoints} $u, v\in \V,u\neq v$ and can be given two orientations $(u,v)$ and $(v,u)$. Two edges $e_1, e_2\in \E$ are {\em adjacent} if they have a common endpoint: $|e_1\cap e_2|=1$. The set of edges {\em incident} to a vertex $v$ consists of all edges $e\in \E$ having $v$ as one of their endpoints. The degree (or valency) of a vertex is the number of edges incident to it.

\medskip
\noindent
A {\em placement}  of $G$ in $\R^d$ is given by a mapping $p: \V\mapsto \R^d$ of the vertices to points in $\R^d$, such that the two endpoints of each edge $e=\{u,v\}\in \E$ are mapped to distinct points in $\R^d$: $p(u)\neq p(v)$.  An edge $\{u,v\}$ is seen geometrically as an {\em edge-segment} $[p(u),p(v)]$, and an oriented edge $(u,v)$ determines an {\em edge-vector} $p(v)-p(u)\in \R^d$. We work under the assumption that {\bf all placements are locally finite maps}, that is, the preimage of any bounded set is finite. This is certainly true for the periodic placements defined in the next section.

\medskip
\noindent
A {\em framework} or {\em geometric graph} $(G,p)$ is a graph $G$ together with a placement $p$, restricted in this paper to $\R^2$ or $\R^3$. We use the term {\em planar placement} for $\R^2$, when the distinction is necessary.  

\medskip
\noindent
\paragraph{\bf Planar non-crossing frameworks.} A planar placement is {\em non-crossing} if any pair of edges induces disjoint closed segments, with the possible exception of the common endpoint, in the case when the edges are adjacent. A graph $G$ is {\em planar}\footnote{Note that we use {\em planar} for the graph, as is customary in graph theory, and {\em non-crossing} for the framework. Our use of {\em planar framework} is customary in rigidity theory, and refers to a placement in the {\em plane.} } 
if it admits a non-crossing placement $(G,p)$. We  consider only connected graphs, therefore a non-crossing placement $(G,p)$ induces a connected subset of the plane, made of the points $p(v), v\in \V$ and the edge segments $[p(u),p(v)]$ with $\{u,v\}\in \E$. A {\em face} $U$ is a connected component of the complement of the placement $(G,p)$. {\bf We assume throughout the paper that  the boundary of each face is a simple finite polygon}. A face is described combinatorially by the cyclic collection of its boundary vertices or edges, and each edge is on the boundary of exactly two distinct faces. 

\medskip
\noindent
For a planar non-crossing framework it is convenient to use now the same symbol $G$ to denote the entire collection $G=(\V,\E,\F)$ of vertices $\V$, edges $\E$ {\em and} faces $\F$.  For notational simplicity, we allow the capital letters $U, V, ...$ to stand for open face domains, closed faces, or the corresponding boundary cycles (as needed in various contexts).  

\medskip
\noindent
When referring to a {\em planar graph} $G$ (with no reference to a particular placement) we assume that the choice of face cycles $\F$ is also given, i.e. $G=(\V,\E,\F)$.  
{\em We note that even when the underlying graph $G=(V,E,F)$ of a framework $(G,p)$ is a planar graph, the particular placement $p$ of the framework  may have crossings: we still refer to the {\em realization of a face}, although it may be a self-intersecting polygon. }

\medskip
\noindent
To a planar graph $G=(\V,\E,\F)$ we can associate a {\em dual} structure $G^*=(\V^*,\E^*, \F^*)$  defined as the abstract triple whose vertices $\V^*$ correspond to the faces $\F$ of $G$ i.e. $\V^* = \F$, and whose edges $\E^*$ are in one-to-one correspondence with the edges $\E$ of $G$, as follows: if two faces  $U$ and $W$  share an edge $e$, then the dual vertices $U^*$ and $W^*$ are connected by the dual edge $e^*$.  The dual faces $\F^*$ correspond to the vertices of $G$ i.e. $\F^*=\V$, with the cycle of faces around a vertex inducing its corresponding dual face.
By abuse of language, we may refer to $G^*=(\V^*,\E^*, \F^*)$ as the {\em dual graph} of $G=(\V,\E,\F)$, although it may have multiple edges.

\medskip
\noindent

\begin{wrapfigure}{l}{0.4\textwidth}
\vspace{-20pt}
\centering
 {\includegraphics[width=0.38\textwidth]{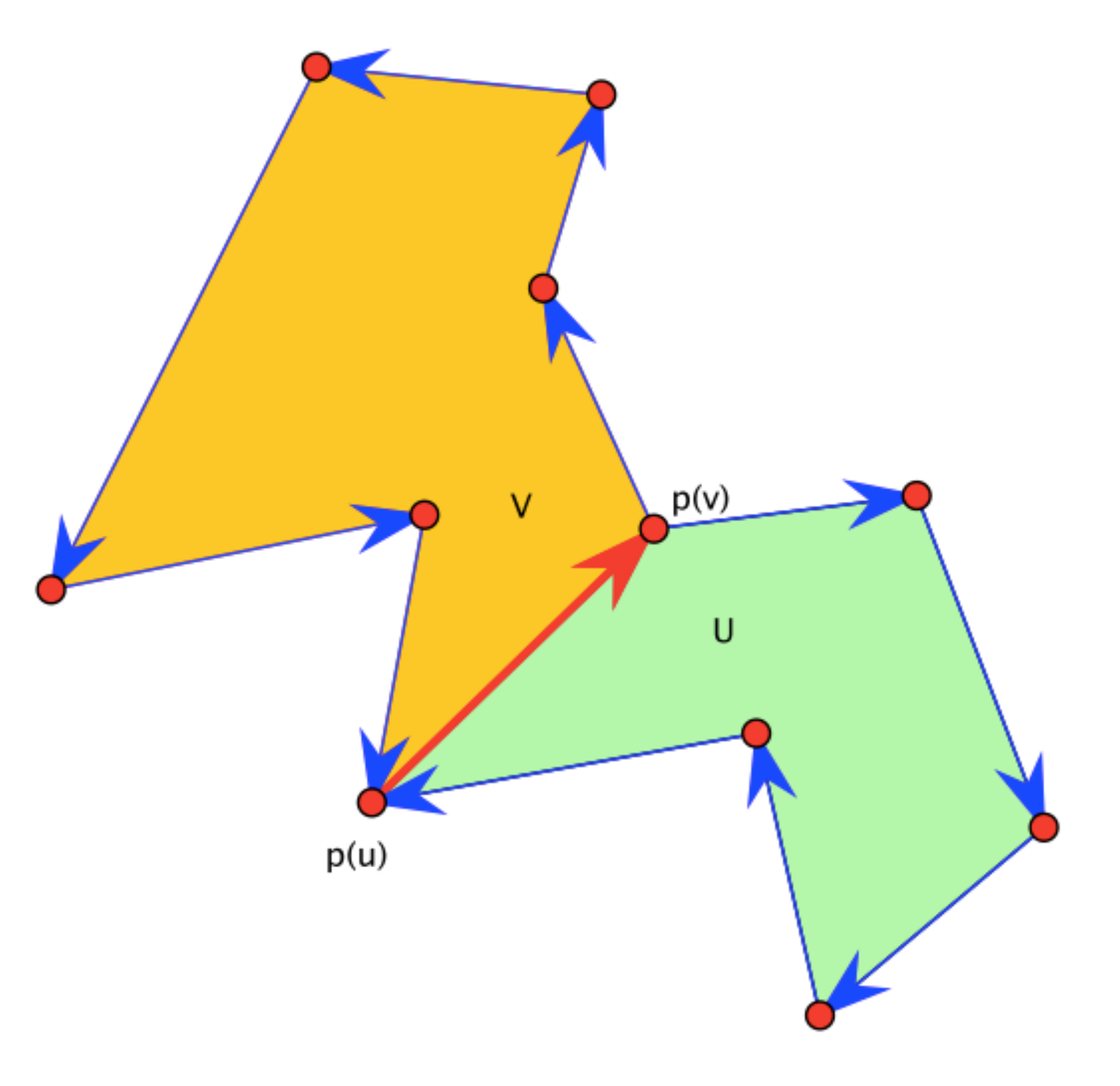}}
\vspace{-4pt}
 \caption{ The orientation convention for face edges and vertex edges.}
\vspace{-14pt}
\label{FigTetrad}
\end{wrapfigure}

\paragraph{\bf Orientation rule.}\  In a planar non-crossing framework, an edge $\{u,v\}$ induces a segment $[p(u),p(v)]$, and it belongs to the boundary of {\em exactly two faces}, say $U$ and $V$. In the dual graph $G^*=(\F,\E)$ these two faces $U$ and $V$ represent two vertices connected by the unoriented edge $\{ U,V\}$  {\em dual}  to $\{ u,v \}$. Later on, we will need to match an {\em oriented} edge in the primal graph $G=(\V,\E)$ with an orientation of its dual edge in $G^*=(\F,\E)$. We use the following convention. The oriented  edge segment $[p(u),p(v)]$ gives opposite senses for going
around face $U$ and face $V$, say counter-clockwise around  $V$ and clockwise around $U$.
Then the matching orientation of the edge in the dual graph $G^*$ is from $U$ to $V$. In short:
`from clockwise to counter-clockwise', as illustrated in Figure~\ref{FigTetrad}.

\medskip \noindent
The matching orientation described above for elements in the edge set $\E$, when considered as oriented edges in $G=(\V,\E)$, respectively $G^*=(\F,\E)$, gives a well-formed double pair $((u,v),(U,V))$, or simply a {\em tetrad} $(u,v,U,V)$. In computational geometry, these tetrads are implicit in the quad-edge data structure used for representing general surfaces \cite{Guibas}. Reversing orientation on the edge gives the tetrad $(v,u,V,U)$. Below, we will refer to cycles (called {\em face-cycles}) of oriented edges in the dual graph $G^*=(F,E)$: the orientation rule described above gives an unambiguous correspondence with {\em oriented edges} $(u,v)$ in the primal graph $G$ and their corresponding  {\em edge vectors} $p(v)-p(u)$ in a geometric placement $p$ of $G$. 

\paragraph{\bf Stressed frameworks.} An {\em equilibrium stress} or, shortly, a {\em stress} on a planar (finite or infinite, possibly crossing) framework\footnote{Also called a {\em self-stress} in the rigidity theory literature.} is an assignment $s: E \rightarrow \R$ of scalar values $\{ s_e \}_{e\in E}$  to the edges $E$ of $G$ in such a way that the edge vectors incident to each vertex $u\in V$, scaled by their corresponding stresses, are in equilibrium, i.e. sum up to zero:

\begin{equation}\label{Scondition}
 \sum_{e=\{ u,v\}\in \E} s_{e} (p(v)-p(u))=0, \ \ \mbox{for fixed}\ u\in \V
\end{equation}

\noindent
When all $s_{e}$ are zero, the stress $s$ is {\em trivial};  when all $s_e\not =0$, the stress is called {\em nowhere zero}. The space of all equilibrium stresses of a framework is a vector space, so if a framework has a non-trivial stress, then it is not unique; in particular, any rescaling of it is also a stress.

\paragraph{\bf Lifting.} A lifting of the planar framework $(G, p)$ is a continuous function $  H: \R^2 \rightarrow \R$ whose restriction to any face is an affine function. The lifting assigns a height, or altitude $H(q)$ to each point $q$ in $\R^2$ (seen as the plane $z=0$ in $\R^3$),  in such a way that the lifted faces are flat (all face cycle vertices lie in the same plane) and connect continuously along the lifted edge segments.  The height function is completely determined by the values $H(p(v))$ at the vertices of the framework, and its graph appears as a polyhedral surface or {\em terrain} over the face-tiling in the reference plane.  A lifting is {\em trivial} if all its faces are lifted in the same plane, that is, when $H$ is affine on $\R^2$. A lifting is {\em strict} if no two adjacent faces are lifted to the same plane.

\medskip
\noindent
With these concepts in place,  we move on to the correspondences involved in Maxwell's theorem.

\paragraph{\bf Equilibrium stress associated to a lifting.} Let $H$ be a lifting of a framework $(G,p)$. With usual dot product notation, the expression of $H$ restricted to a face $U$ takes the form 
\begin{equation}
	\label{Hform}
H(q)=\nu_U\cdot q + C_U, \ \ {\text{for}} \ q\in U \subset \R^2
\end{equation}

\noindent
where $\nu_U\in \R^2$ is the projection on the reference plane of the {\em normal} to face $U$ and $C_U\in \R$.

\medskip \noindent
The system of vectors and constants $H\equiv (\nu_U,C_U)_{U\in \F}$ is subject to the 
compatibility conditions on edges $\{u,v\}$ shared by pairs of adjacent faces $\{U,V\}$:

$$ \nu_U\cdot p(u) +C_U=\nu_V\cdot p(u) + C_V  $$
\begin{equation}
	\label{compat}
\nu_U\cdot p(v) +C_U=\nu_V\cdot p(v) + C_V 
\end{equation}

\medskip
\noindent
From this, we infer that the vector $\nu_V-\nu_U$ is orthogonal to the edge vector $p(v)-p(u)$. Fig.~\ref{fig:Signs} illustrates the relationship.

\begin{wrapfigure}{l}{0.34\textwidth}
\vspace{-22pt}
\centering
{\includegraphics[width=0.33\textwidth]{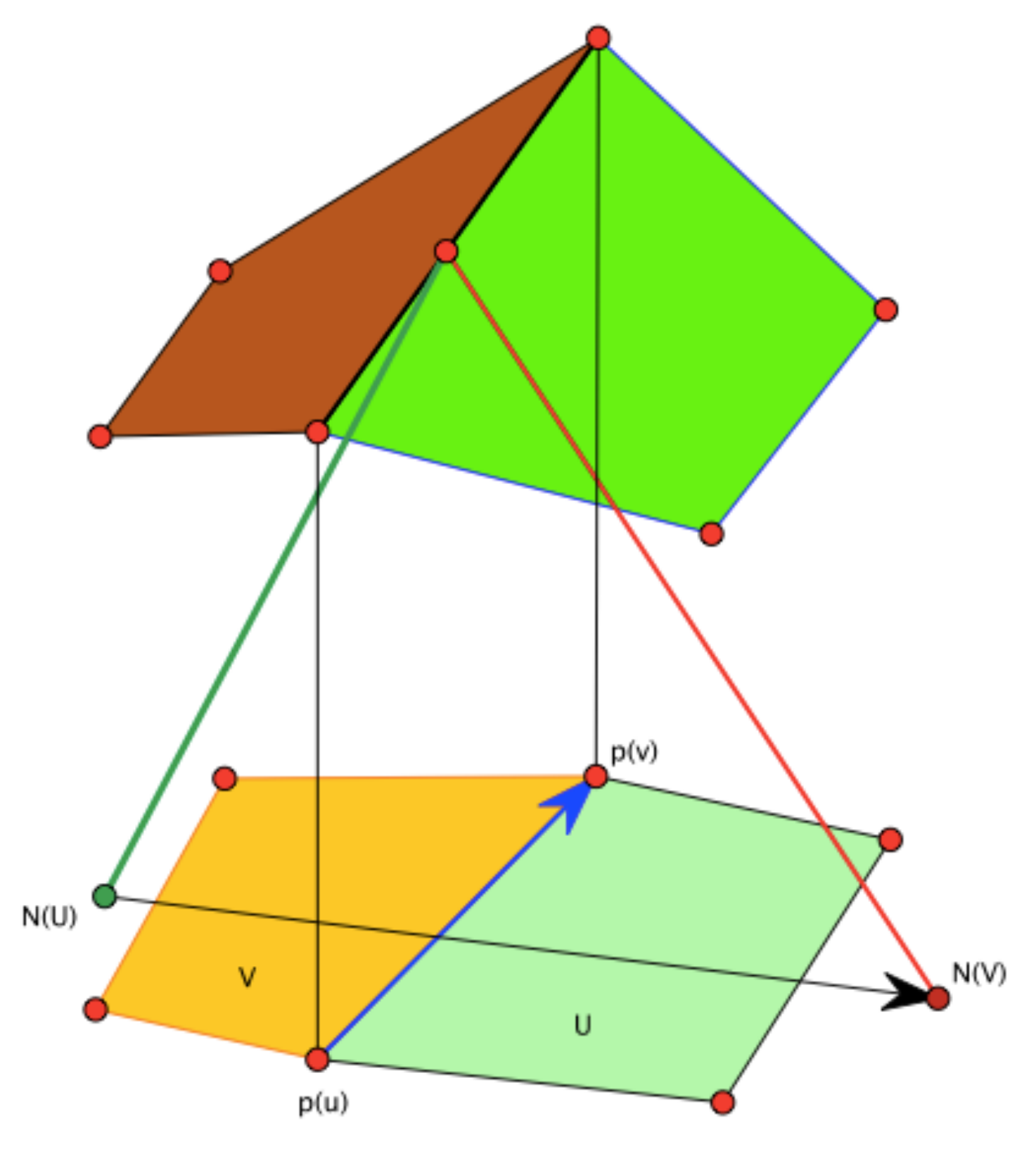}}
\vspace{-4pt}
\caption{The normals to two adjacent lifted faces induce the dual orthogonal edge.}
\vspace{-16pt}
\label{fig:Signs}
\end{wrapfigure}

\medskip \noindent
Given a tetrad $(u,v,U,V)$ of dual edges, and using the notation $(x,y)^{\perp}=(-y,x)$ for the clockwise rotation with $\pi/2$ of a vector $(x,y)\in \R^2$, we define the {\em stress factor} on the edge $\{ u,v\}\in E$, associated to the lifting $H$, as being the proportionality factor $s_{uv}$ given by:

\begin{equation}\label{factor}
\nu_V-\nu_U=s_{uv}(p(v)-p(u))^{\perp}
\end{equation}

\medskip \noindent
Since the sum involves the vectors around a closed polygon (the face-cycle around a vertex), the equilibrium condition (\ref{Scondition}) is satisfied. This proves:

\begin{prop}\label{HtoS}
For any lifting $H$ of a planar non-crossing framework $(G,p)$, there exists a canonically associated equilibrium stress on the framework.
\end{prop} 

\noindent
The correspondence between liftings and stresses described above is essentially the one given by Maxwell, who formulated it through the following geometric construction. The normal direction to the planar region corresponding to a face $U\in \F$ in the lifted terrain is given by $  N_U=(\nu_U,-1)\in \R^3, \  U\in \F$. When all these normal vectors are taken through the point $(0,0,1)$, they intersect the reference plane $z=0$ in the system of points $\{\nu_{U}\in \R^2\}_{U\in F}$.  The classical ``theorem of the three perpendiculars'' implies the orthogonality  $(\nu_V-\nu_U)\cdot (p(v)-p(u)) =0$ observed above.  In Figure~\ref{fig:Signs} we see the normals to two adjacent lifted faces, taken from a point of the lifted common edge. They intersect the reference plane in $N(U)$ and $N(V)$, with vectors $p(v)-p(u)$ and $N(V)-N(U)$ orthogonal.

\paragraph{\bf Reciprocal diagram.} A framework $(G^*,p^*)$ associated to the dual graph $G^*$ of a planar framework $(G,p)$ is called a {\em reciprocal diagram} if the corresponding primal-dual edges are perpendicular. If in the previous construction we join the points $\{\nu_U\in \R^2\ | \ U\in \F\}$ by edges dual to the primal ones, we obtain a reciprocal diagram associated to the lifting $H$. We note that it is possible for several vertices $\nu_U$ to coincide, and this happens precisely when several planar regions in the lifting have identical normal directions. An extreme case arises for liftings with globally affine functions $H$. They give a planar (trivial) terrain over the reference plane, have constant $\nu_U$ and induce the trivial stress $\{s_{e}\}_{e\in E}=0$.

\paragraph{\bf From equilibrium stresses to liftings.} The direction from stresses to Maxwell liftings requires more work.

\begin{prop}\label{StoH}
Let $s=(s_e)_{e\in \E}$ be an equilibrium stress for the framework $(G,p)$. Then there exists a lifting $H$ which induces $s$, determined up to addition of a global affine function.
\end{prop} 

\medskip \noindent
\begin{proof} We have to find a set of parameters $(\nu_U,C_U)$, indexed by faces and satisfying the conditions (\ref{compat}) and (\ref{factor}) in terms of the given placement $p$. Let us choose an initial face $U_0$ with an arbitrary lifting $(\nu_{U_0},C_{U_0})=(\nu_0,C_0)$. We show that once this initial choice has been made, the lifting is then uniquely determined by the placement $p$ and the stress values $s$.

\medskip \noindent
For this purpose, we solve the linear system (\ref{compat}) in a step-by-step manner, progressing from face to adjacent face, starting at $U_0$. We remark that an infinite graph induces an infinite such linear system, but nevertheless it can be solved incrementally, as we now show. We consider a {\em path} through adjacent faces labeled $U_0,U_1,...,U_n$,
with corresponding liftings $(\nu_i,C_i)$ and successive tetrads $(p_i,q_i,U_i,U_{i+1})$. Thus,
$[p_i,q_i]$ is the common edge between faces $U_i$ and $U_{i+1}$, with the proper orientation induced by the direction in which we walk through the path of adjacent faces, and the orientation rule described earlier.
The given stress on this edge is denoted here by $s_i$. 

\medskip \noindent
We first compute the parameters $\nu_k$ by unfolding the relationships (\ref{factor}) along the path from $U_0$ to $U_k$: 

\begin{equation}
	\label{recursion}
\nu_{k+1}= \nu_k + s_k(q_k-p_k)^{\perp} = \cdots = \nu_0 + \sum_{i=0}^k s_i(q_i-p_i)^{\perp}
\end{equation}

\medskip \noindent
Similarly, we compute the parameters $C_k$ by unfolding the first of the two relationships (\ref{compat}) along the path from $U_0$ to $U_k$, and using (\ref{factor}) at each step: 

\begin{equation}
	\label{recursionC}
C_{k+1}=C_k - (\nu_{k+1}-\nu_k)\cdot p_k= \cdots = C_0 -\sum_{i=0}^k s_i (q_i-p_i)^{\perp}\cdot p_i
\end{equation}

\medskip \noindent
Using  the identity $ (q_i-p_i)^{\perp}\cdot p_i =det(q_i\ p_i)$ and the notation $|q_i\ p_i|:=det(q_i\ p_i)$ in (\ref{recursion}) and (\ref{recursionC}), the expression (\ref{Hform}) of the height function becomes:
\begin{equation}\label{path}
H(q)=\nu_n\cdot q +C_n=(\nu_0+\sum_{i=0}^{n-1} s_i(q_i-p_i)^{\perp})\cdot q +
(C_0-\sum_{i=0}^{n-1} s_i|q_i\ p_i|)
\end{equation}

\noindent
It remains to check that the expression (\ref{path}) is independent of the face-path chosen from $U_0$ to $U_n$. 
To verify this property, we have to check that the following sums vanish for any face-cycle:

\begin{equation}\label{Fcycles}
\sum_{i\in face-cycle} s_i(q_i-p_i) =0  \ \  and \ \ \sum_{i\in face-cycle} s_i|q_i\ p_i|=0
\end{equation}

\medskip \noindent
It suffices to verify these relations over face-cycles corresponding to simple topological loops. In this case, Jordan's simple curve theorem gives a set of vertices inside the loop. When we sum the equilibrium stress condition (\ref{Scondition}) over the vertices inside the loop, terms cancel in pairs for adjacent vertices and yield the first identity. For the second part, we observe that the equilibrium stress condition (\ref{Scondition}) induces the identity $   | \sum_{\{ u,v\}\in \E} s_{uv} (p(v)-p(u))\  p(u)| =0, \ \ \mbox{for fixed vertex} \ u $. Rewritting it as $\sum_{\{ u,v\}\in \E} s_{uv} |p(v)\ p(u)|=0$ and again taking the sum over the vertices inside the loop, we obtain the second identity in (\ref{Fcycles}). 

\medskip
\noindent
Since the initial choice $(\nu_0,C_0)$ was arbitrary, the lifting $H$ is determined only up to a global affine function.  
\end{proof}

\paragraph{\bf Mountain-valley edge liftings and stress factor signs.}
We review now an elementary geometrical criterion for determining the sign of a stress factor $s_{uv}$ in a stress $s$ associated to a lifting.  Let us consider the lifted 3D polygonal surface projecting over the faces of the graph in the reference plane, as in Figure~\ref{fig:liftFiniteGraph}, and let $uv$ be part of a tetrad $(u,v,U,V)$. A vertical translation of the whole polygonal surface does not affect the stress considerations. We assume therefore that the lifting of the two faces $U$ and $V$ is above the reference plane. Figure~\ref{fig:Signs} illustrates the case of a negative stress factor $s_{uv}$, corresponding visually with a lifted edge that looks like a `mountain ridge'. The normal directions to the lifted faces over $U$ and $V$ are here taken through a point of the lifted common edge and intersect the reference plane in points denoted $N(U)$ and $N(V)$. For the tetrad $(u,v,U,V)$, the sign of the stress factor is that of the determinant $|p(v)-p(u)\ N(V)-N(U)|$ and is consistent with the orientation rule described in Section \ref{sec:preliminaries}. 

\medskip \noindent
For the geometrical interpretation, let us  look at the {\em dihedral angle} of the lifted faces, understood as the intersection of the two half-spaces below the respective supporting planes. Then, when the measure of this dihedral angle is less than $\pi$, we have a negative stress factor and the visual landscape suggestive of a mountain range, 
while for a dihedral angle greater than $\pi$, we have a positive stress factor and the visual landscape suggestive of a valley. When the dihedral angle is exactly $\pi$, the stress factor is zero and the two lifted faces are in the same plane. A non-flat edge in the lifting is therefore said to be a {\em mountain} edge if the terrain is concave in its neighborhood, and a {\em valley} otherwise. The correspondence between stresses and liftings is now refined by this well-known property \cite{CW}: 

\begin{prop}
\label{prop:signStress}
The Maxwell correspondence between stressed graphs and liftings takes planar edges with a negative stress to mountain edges in the 3D lifting, those with positive stress to valley edges and those with zero stress to the common plane of the two adjacent lifted faces. 
\end{prop}

\begin{figure}
\centering
\subfigure[]{\includegraphics[width=0.2\textwidth]{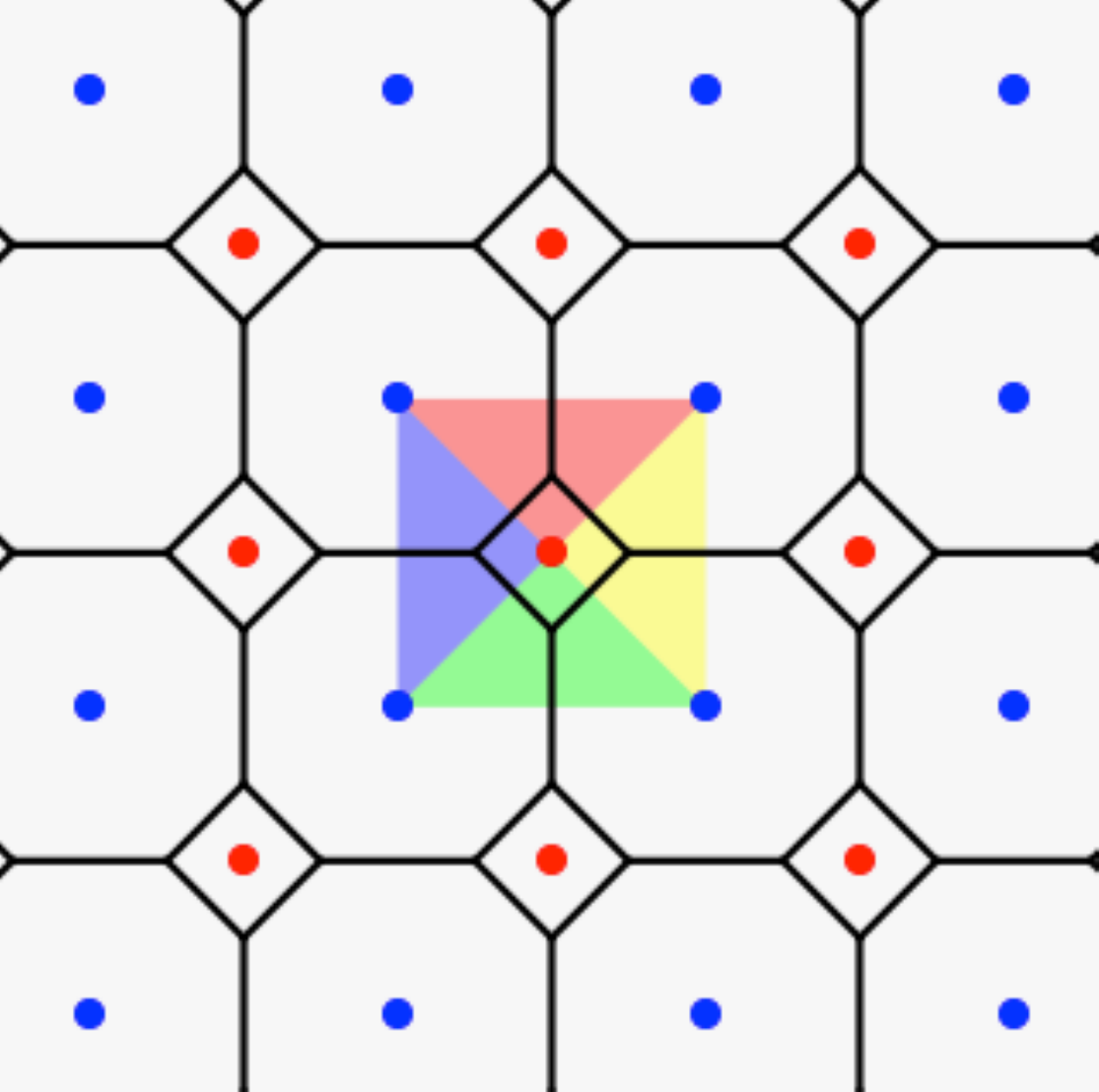}}
\hspace{6pt}
\subfigure[]{\includegraphics[width=0.39\textwidth]{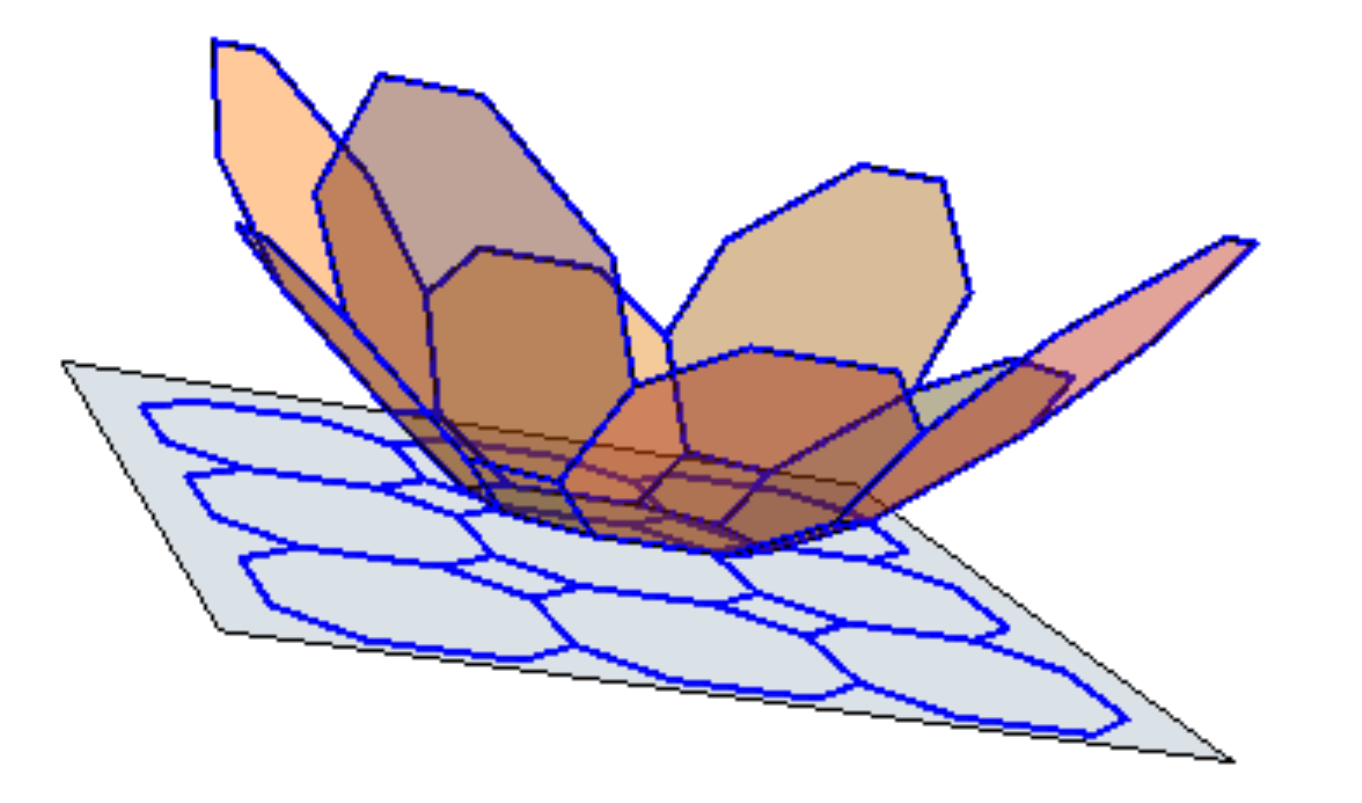}}
\subfigure[]{\includegraphics[width=0.2\textwidth]{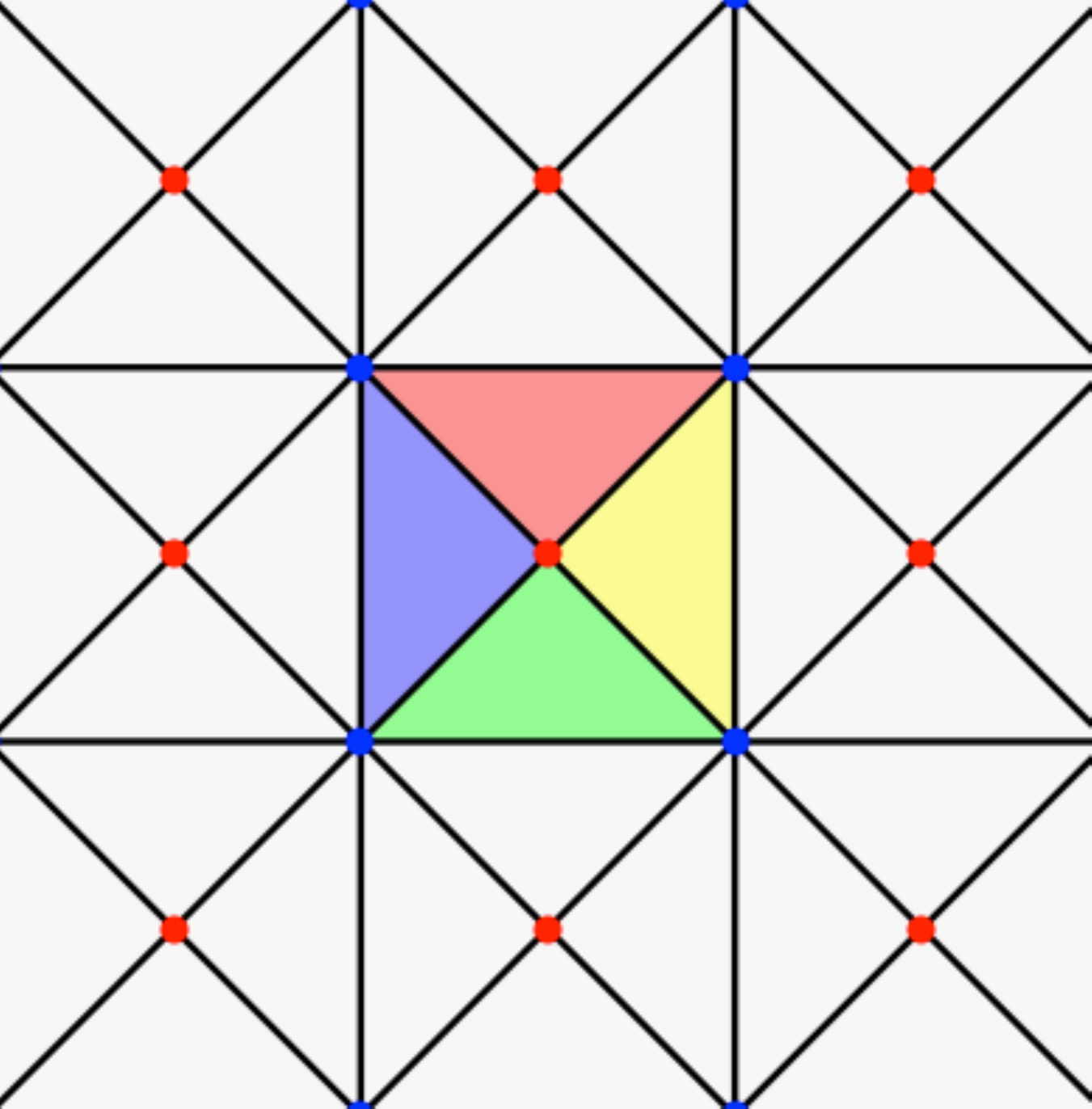}}
 \caption{ (a) A Delaunay framework (in black) corresponding  to a periodic set with four site orbits under maximal translational symmetry. Dual vertices are shown as colored centers of the
faces. Primal-dual edge pairs are orthogonal. (b) The equilibrium stress induced by the reciprocal diagram has a non-periodic lifting to the paraboloid $z=x^2+y^2$. (c) The reciprocal diagram of the periodic graph on the left is the Voronoi diagram of the periodic sites.}
 \label{fig:VoronoiDelauney}
\end{figure}

\medskip \noindent
{\bf Comment.}\ As mentioned, this correspondence between liftings and  stresses is a direct adaptation of the one formulated by Maxwell in the finite setting \cite{M1,M2} (see also the expositions given in \cite{hopcroft:kahn:paradigmRobustAlg:1992,CW,RG}). The arrangement and presentation given in this section, including the notation, are meant to serve as a preamble for the more elaborate  {\em periodic} version, which relies on notions developed recently in our deformation theory of periodic frameworks \cite{BS2,BS3,BS4}. 

\medskip
\noindent
We conclude this section with a noteworthy example, illustrated in Fig.~\ref{fig:VoronoiDelauney}, which will be used in the next two sections to illustrate the critical distinction between equilibrium stress and periodic stress.
\begin{examp}{\bf \  (Delaunay tesselations and Voronoi diagrams of periodic point sets)}\  \ 
The classical Voronoi-Delauney duality, applied to a (countable locally finite) periodic point set $p$ in $\R^2$, yields a dual pair of non-crossing periodic frameworks whose corresponding dual edges are orthogonal.  For a fixed point $p_i$ in the given set $p$, its (open) Voronoi cell is defined as the planar domain made of all the points in $\R^2$ whose distance to $p_i$ is strictly less than the distance to any other point in the set $p$. The union of all these cell boundaries give the vertices and edges of the Voronoi framework. The Delauney framework is its dual, with vertices at the original point set $p$ and with edges orthogonal to those of the Voronoi framework; its faces are convex polygons whose vertices are cocircular, and the circumscribed cirles are empty of any other vertices in the given point set \cite{delaunay:SphereVide:1934}. Fig.~\ref{fig:VoronoiDelauney} gives an illustration with periodicity lattice generated by the standard basis. There are four orbits of vertices, with representatives at the corners of one of the diamonds in Fig.~\ref{fig:VoronoiDelauney}(a). All faces are inscribed polygons and the centers of the corresponding circumscribed circles give the vertices of the Voronoi diagram.
\end{examp}

\section{Equilibrium and periodic stresses on periodic frameworks} 
\label{sec:periodicLiftingStress}

We turn now to the main object studied in this paper, the infinite {\em periodic framework} as defined in \cite{BS2,BS3}. We emphasize from the outset that a periodic graph is not just an infinite graph, as it was in the previous section: the definition {\em includes} a periodicity group of graph automorphisms. In this paper we focus on the {\em planar case} $d=2$ and on {\em connected non-crossing periodic frameworks}. Some statements and constructions will be valid in broader contexts, but for the sake of a streamlined presentation we will stay within this class.

\paragraph{\bf Planar periodic frameworks.}
A 2-periodic framework \cite{BS2,BS3}, denoted as $(G,\Gamma, p,\pi)$, is given by an infinite graph $G$, a periodicity group $\Gamma$ acting on $G$, a placement $p$ and a representation $\pi$. The graph $G=(\V,\E)$ is simple (has no multiple edges and no loops) and connected, with an infinite set of vertices $\V$  and (unoriented) edges $\E$. The {\em periodicity group} $\Gamma\subset Aut(G)$ is a free Abelian group of rank two acting on $G$ without fixed vertices or fixed edges.  We consider only the case where the quotient  multigraph $G/\Gamma$ (which may have loops and multiple edges) is {\em finite}, and use $n=|\V/\Gamma|$  and $m=|\E/\Gamma|$ to denote the number of vertex and edge orbits.  The function $p:\V\rightarrow \R^2$ gives a specific placement of the vertices as points in the plane, in such a way that any two vertices joined by an edge in $\E$ are mapped to distinct points. The {\em injective group morphism} $\pi: \Gamma\rightarrow  {\cal T}(\R^2)$ gives a faithful representation of $\Gamma$ by a {\em lattice of translations} $\pi(\Gamma)=\Lambda$ of rank two in the group of  planar translations ${\cal T}(\R^2)\equiv \R^2$.  The placement is {\em periodic} in the obvious sense that the abstract action of the periodicity group $\Gamma$ is replicated by the action of the periodicity lattice $\Lambda=\pi(\Gamma)$
on the placed vertices: $p(\gamma v)=\pi(\gamma) (p(v)), \ \ \mbox{for all}\ \gamma\in \Gamma, v\in \V$.

\paragraph{\bf Non-crossing planar periodic frameworks $G=((\V,\E,\F),\Gamma)$ in  $\R^2$.} They have an underlying planar graph $G=(\V,\E,\F)$ with the natural action of the periodicity group $\Gamma$.
The {\em dual 2-periodic  graph} $G^*=((\V^*,\E^*, \F^*),\Gamma)$ is obtained from the abstract dual of the infinite graph $G=(\V,\E,\F)$, with the periodicity group $\Gamma$ acting on it in the same manner as it acts on the sets of the primal graph $\F=\V^*, \E=\E^*, \V=\F^*$. 
 If we denote by $n^*=card(\F/\Gamma)$ the number of face orbits under $\Gamma$, then
Euler's formula for the torus $\R^2/\Lambda \supset G/\Gamma$ gives the relation: 

\begin{equation}\label{Euler}
n-m+n^*=0, \ \ \ \mbox{that is}\ \ \  n+n^*=m.
\end{equation}

\begin{figure}[h]
\centering
\subfigure[]{\includegraphics[width=0.29\textwidth]{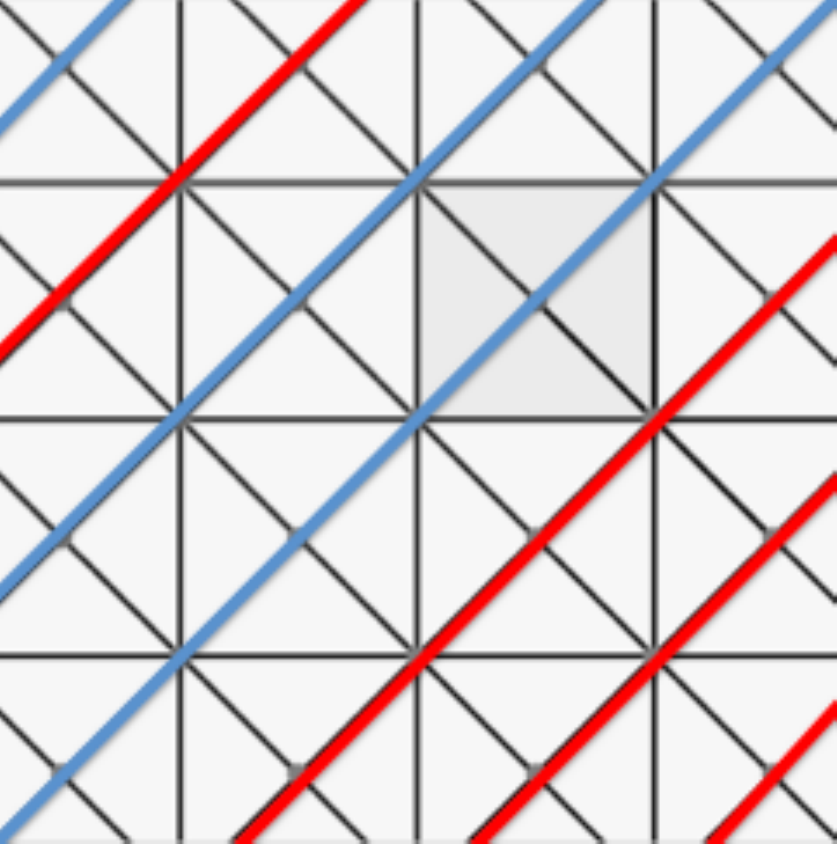}}
\hspace{6pt}
\subfigure[]{\includegraphics[width=0.29\textwidth]{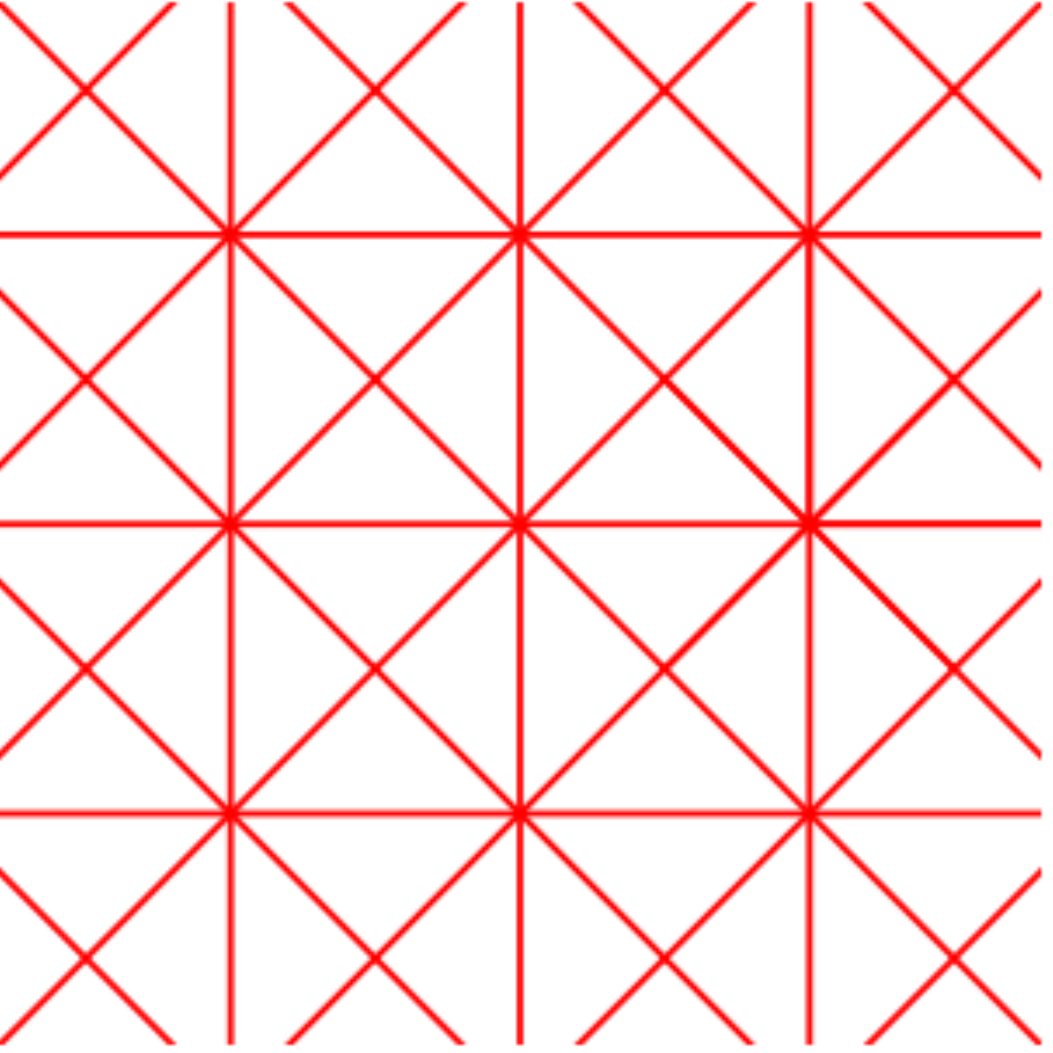}}
\hspace{6pt}
\subfigure[]{\includegraphics[width=0.29\textwidth]{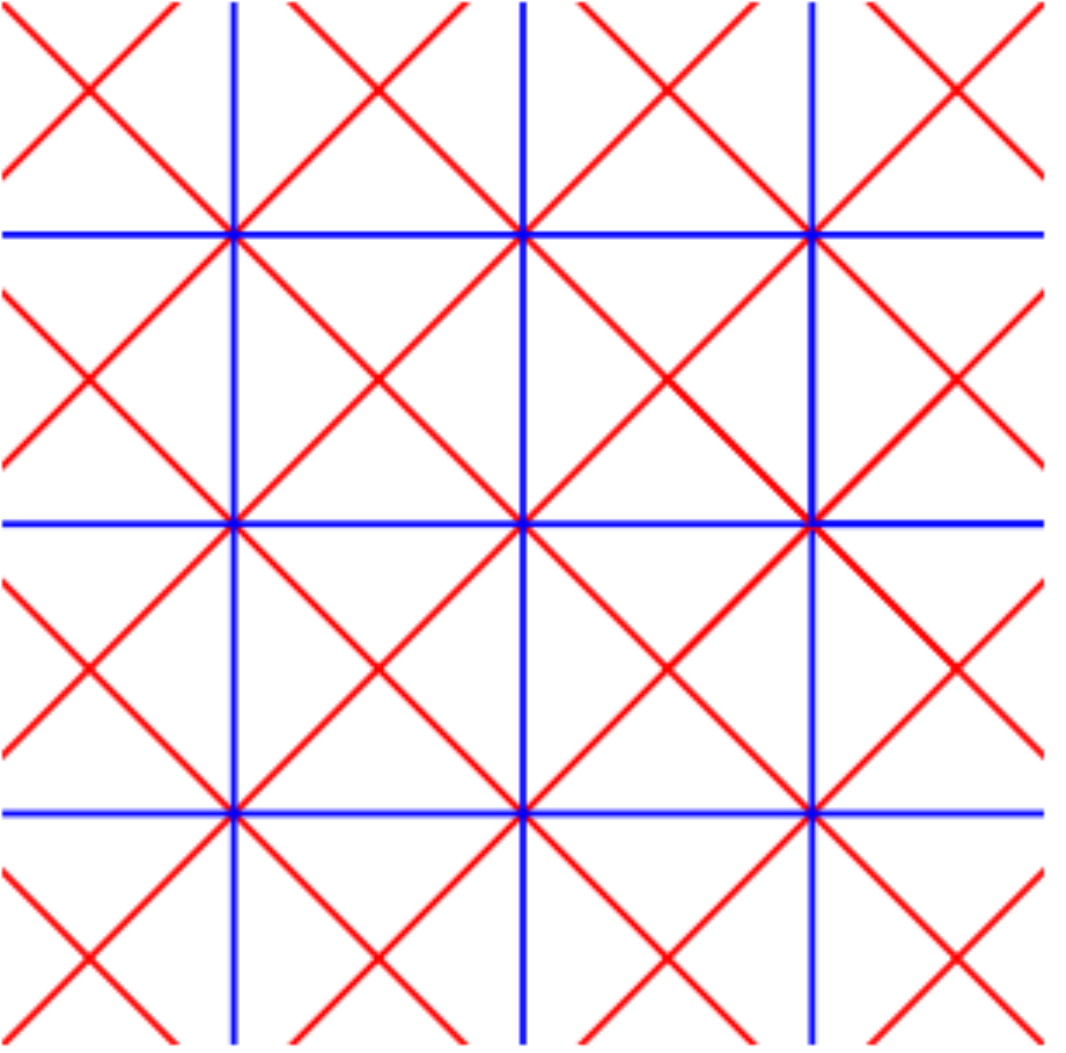}}
\vspace{-6pt}
 \caption{The Voronoi diagram from Fig.~\ref{fig:VoronoiDelauney}(c) is a periodic framework supporting different types of stresses. (a) {\bf (Not $\Gamma$-invariant)} Each of the ``aligned'' infinite paths supports a one-dimensional stress (equal on each edge of the path and illustrated here with the oblique colored diagonals). The stress values can be independently chosen on each diagonal path, and thus yield equilibrium stresses, such as the one depicted here, which are not $\Gamma$-invariant. (b) {\bf (Non-periodic $\Gamma$-invariant)} A $\Gamma$-invariant stress assigns the same stress value on all edges in an edge orbit. Illustrated here is a stress where all stress orbits have the same sign; this cannot be a periodic stress. (c) A {\bf periodic stress}, as the one shown here, must have both positive and negative stresses on edge orbits.}
\vspace{-16pt}
 \label{fig:TypesOfStresses}
\end{figure}

\paragraph{\bf $\Gamma$-invariant equilibrium stress on a periodic framework.} An equilibrium stress $s$ of the  periodic framework $(G, p,\Gamma,\pi)$ is called a {\em $\Gamma$-invariant equilibrium stress} if it is invariant on edge orbits $\E/{\Gamma}$. A $\Gamma$-invariant equilibrium stress can be calculated by solving a finite linear system of equations of type (\ref{Scondition}), where the unknowns are the stresses for the edge representatives in $\E/{\Gamma}$ and the equations correspond to equilibrium conditions for the vertex representatives in $\V/{\Gamma}$.

\paragraph{\bf Preview: periodic stress from periodic lifting.} 
Not all equilibrium stresses of a periodic framework are $\Gamma$-invariant. Although we have not yet defined periodic stresses (Definition (\ref{Pstress}) in the next Section \ref{sec:deformationPeriodicStress}), we mention, as a preview, that they must be $\Gamma$-invariant. But this will not be sufficient: for some periodic frameworks, the periodic stresses are a {\em strict} subset of the $\Gamma$-invariant ones.  Figure~\ref{fig:TypesOfStresses} illustrates such a situation, using colors to indicate the stress sign (red for positive, blue for negative stress, gray for zero stress) for the various types of stresses that a periodic framework may support (not $\Gamma$-invariant, $\Gamma$-invariant non-periodic and periodic).

\medskip
\noindent
We will arrive at periodic stress indirectly, via periodic liftings, defined next. In the rest of this section we show that the stresses corresponding to these periodic liftings must satisfy additional constraints. Then, in the next section, we define periodic stress in rigidity theoretic terms, and show that it coincides with this constrained stress.

\begin{wrapfigure}{l}{0.36\textwidth}
\vspace{-8pt}
\centering
{\includegraphics[width=0.36\textwidth]{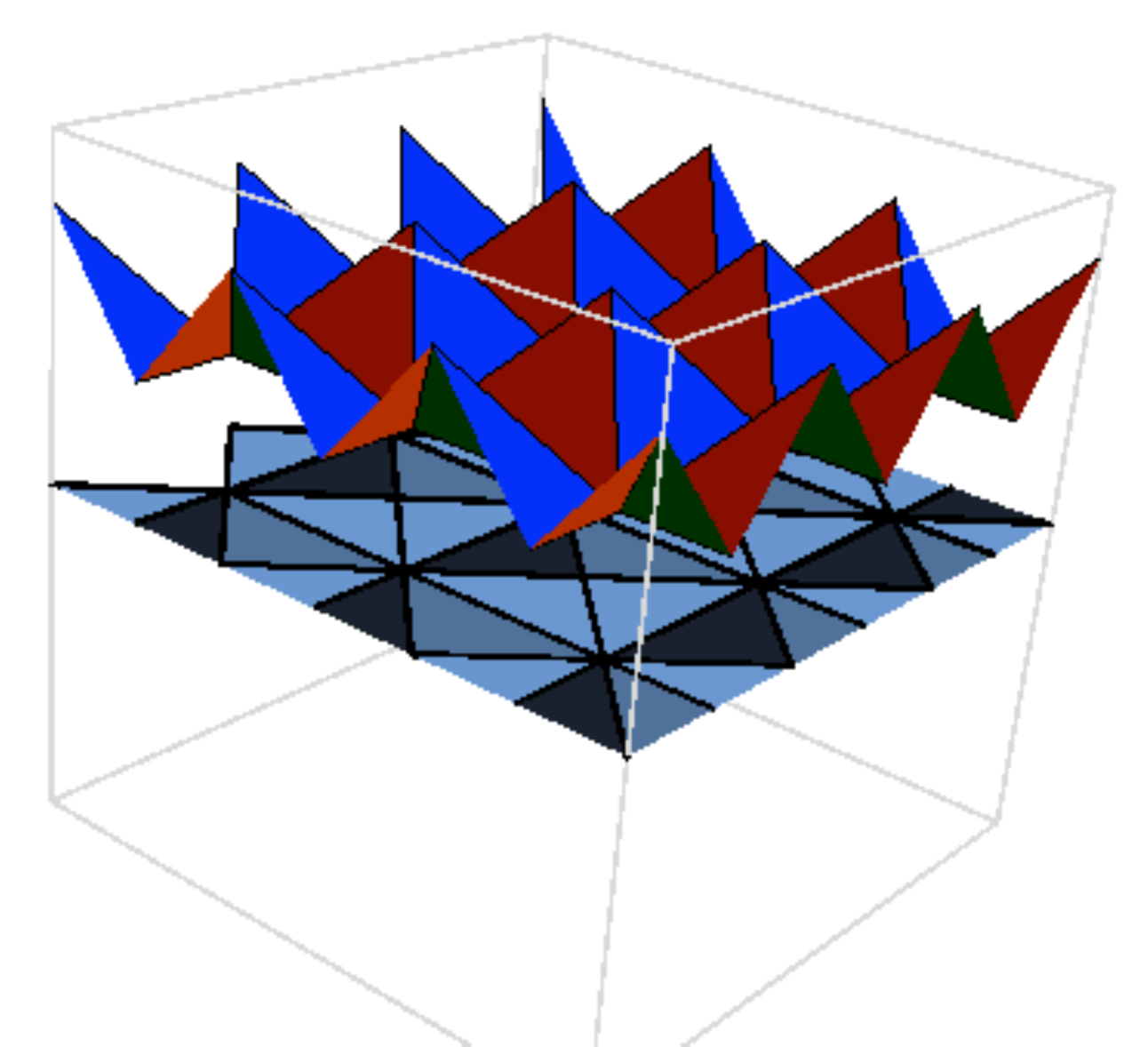}}
 \caption{A periodic 3D lifting for the stressed framework in Figure~\ref{fig:TypesOfStresses}(c).}
 \label{fig:lift3D}
\end{wrapfigure}

\paragraph{\bf Periodic liftings.} A lifting $H$ for the planar periodic framework $(G,\Gamma,p,\pi)$ is called periodic when $\Gamma$-invariant, that is, when $H(p+\lambda)=H(p)$, for all $p\in \R^2$ and $\lambda\in \Lambda=\pi(\Gamma)$, where  translation by periods has been written additively. For a face $U$ and its translate $U+\lambda$, this condition implies that:
$$\nu_{U+\lambda} = \nu_U$$
\begin{equation}\label{perLiftingCond}
C_U-C_{U+\lambda}=\nu_U \cdot \lambda
\end{equation}

\medskip
\noindent
Fig.~\ref{fig:lift3D} illustrates a periodic lifting which corresponds to the stress illustrated in Fig.~\ref{fig:TypesOfStresses}(c).

\medskip
\noindent
Since $\Gamma$-invariant stresses may not always correspond to  $\Gamma$-invariant liftings,  additional properties are needed to characterize stresses induced by $\Gamma$-invariant liftings. We proceed now to find them. 

\paragraph{\bf Stress induced by a periodic lifting.}
When expressed as the system $H\equiv (\nu_U,C_U)_{U\in \F}$, a periodic lifting
has constant coefficients $\nu_U$ on $\Gamma$-orbits of faces. Thus, there are at most $n^*=card(\F/\Gamma)$ distinct normal directions to lifted faces.

\medskip 
\noindent
For a closer investigation of the associated stress, we introduce the following notational conventions. A face-path from a face $U$ to a face $V$ will be indicated by $U\rightarrow V$, in particular, a face-path from $U$ to its translate $U+\lambda$ will be indicated by
$U\rightarrow U+\lambda$. Sums over face-paths or face-cycles are assumed to be written
according to the orientation rule given through tetrads.
With this convention, we rewrite the relations (\ref{recursion}) and (\ref{recursionC}) obtained in Section \ref{sec:preliminaries} as:

$$ \nu_V=\nu_U + \sum_{U\rightarrow V} s_i(q_i-p_i)^{\perp} $$

\begin{equation}\label{recursionBis}
C_V=C_U -  \sum_{U\rightarrow V} s_i |q_i\ p_i|
\end{equation}

Applying (\ref{perLiftingCond}) we obtain the following two conditions on the stress $s$:

\begin{equation}\label{Pnecessary}
 \sum_{U\rightarrow U+\lambda} s_i(q_i-p_i) =0 
\end{equation}

\begin{equation}\label{Cnecessary}
\sum_{U\rightarrow U+\lambda} s_i|q_i\ p_i|= C_U-C_{U+\lambda}=\nu_U\cdot \lambda
\end{equation}

\noindent
We summarize these observations as:
\begin{prop}\label{PLnecessary}
If the planar non-crossing periodic framework $(G,\Gamma,p,\pi)$ has a periodic lifting $H\equiv (\nu_U,C_U)_{U\in \F}$, then the
associated $\Gamma$-invariant equilibrium stress $s$ satisfies conditions (\ref{Pnecessary}) and (\ref{Cnecessary})
for any face $U$ and period vector $\lambda\in \Lambda=\pi(\Gamma)$.
\end{prop}

\paragraph{\bf From constrained stress to periodic lifting.}
Our next goal is to show that conditions (\ref{Pnecessary}) on a $\Gamma$-invariant equilibrium stress are sufficient for determining a periodic lifting. Obviously, relations (\ref{Cnecessary}) will serve for identifying the lifting data $H\equiv (\nu_U,C_U)_{U\in\F}$. 

\begin{prop}\label{StoPL}
Let $s=(s_e)_{e\in \E}$ be a $\Gamma$-invariant equilibrium stress for the planar non-crossing periodic
framework $(G,\Gamma,p,\pi)$. If, for some face $U_0$ and generators $\lambda_1, \lambda_2$ of the period lattice $\Lambda=\pi(\Gamma)$, $s$ satisfies the additional conditions that:

\begin{equation}\label{PLcond}
 \sum_{U_0\rightarrow U_0+\lambda_j} s_i(q_i-p_i) =0, \ \ \ j=1,2  
\end{equation}

\noindent
then the lifting $H\equiv (\nu_U,C_U)$ defined, for all $\lambda\in \Lambda$,  by:

$$ \nu_U\cdot \lambda =\sum_{U\rightarrow U+\lambda} s_i|q_i\ p_i|
=C_U-C_{U+\lambda} $$

\noindent
is a periodic lifting inducing $s$, determined up to a choice of constant $C_0=C_{U_0}$.
\end{prop}

\begin{proof}
Let us {\em assume} that $s=(s_i)_i$ is a $\Gamma$-invariant equilibrium stress which satisfies (\ref{PLcond}) for some face $U_0$ and two generators $\lambda_1, \lambda_2$ of the periodicity lattice $\Lambda=\pi(\Gamma)$. We first show that our assumption implies that condition (\ref{Pnecessary}) is satisfied for all $U$ and $\lambda\in \Lambda$. It is clearly satisfied for $U_0$ and all $\lambda\in \Lambda$ since the edge vectors implicated in the sum are, by periodicity, the same linear
combination with integer coefficients as given in $\lambda=n_1\lambda_1+n_2\lambda_2$.

\medskip \noindent
Suppose next that $V$ is adjacent to $U_0$ and consider the
face-cycle $U_0\rightarrow V\rightarrow V+\lambda \rightarrow U_0+\lambda \rightarrow U_0$. Again, by periodicity
in the framework, the edge vector implicated for $U_0\rightarrow V$ is the opposite of
the edge vector implicated for $V+\lambda\rightarrow U_0+\lambda$ and (\ref{Pnecessary})
holds for $V$. This is enough for our claim since $G$ and $G^*$ are connected.  

\medskip \noindent
As shown in the proof of Prop.~\ref{StoH}, for any initial choice $(\nu_{U_0},C_{U_0})=(\nu_0,C_0)$, we obtain a lifting $H\equiv (\nu_U,C_U)$ inducing $s$.
We aim at finding a periodic representative in this equivalence class.

\medskip \noindent
Let $H\equiv (\nu_U,C_U)$ be the lifting under investigation. Relations (\ref{recursionBis})
will hold. Thus $\nu_U$ are constant on $\Gamma$ orbits in $\F$. We show now that the function

$$ {\cal N}_U(\lambda)=\sum_{U\rightarrow U+\lambda} s_i|q_i\ p_i| $$

\noindent
is independent of the orbit representative $U$. Indeed, using (\ref{Pnecessary}) we find

$$ \sum_{(U+\mu)\rightarrow (U+\mu)+\lambda} s_i|q_i+\mu\ p_i+\mu|= $$

\begin{equation}\label{indep}
=\sum_{U\rightarrow U+\lambda} s_i|q_i\ p_i| +
\sum_{U\rightarrow U+\lambda} s_i|q_i- p_i\ \mu| =
\sum_{U\rightarrow U+\lambda} s_i|q_i\ p_i|
\end{equation}

\medskip \noindent
Moreover, ${\cal N}_U(\lambda)$ is linear in $\lambda$ and by (\ref{recursionBis}):

\begin{equation}\label{goodN}
{\cal N}_U(\lambda)=C_U-C_{U+\lambda}
\end{equation}

\medskip \noindent
By comparing $C_V-C_U$ and $C_{V+\lambda}-C_{U+\lambda}$ based on (\ref{recursionBis}),
we obtain

$$ C_{V+\lambda}-C_{U+\lambda}=-\sum_{U\rightarrow V} s_i|q_i+\lambda\ p_i+\lambda|=
C_V-C_U -\sum_{U\rightarrow V} s_i|q_i-p_i\ \lambda|= $$

$$ =C_V-C_U-\sum_{U\rightarrow V} s_i(q_i-p_i)^{\perp}\cdot \lambda=
C_V-C_U-(\nu_V-\nu_U)\cdot \lambda $$

\medskip \noindent
This gives

$$ {\cal N}_V(\lambda)=C_V-C_{V+\lambda}=C_U-C_{U+\lambda} +(\nu_V-\nu_U)\cdot \lambda=
{\cal N}_U(\lambda)+(\nu_V-\nu_U)\cdot \lambda $$

\noindent
and shows that

\begin{equation}\label{Nrelation}
 {\cal N}_V(\lambda) -  {\cal N}_U(\lambda)=(\nu_V-\nu_U)\cdot \lambda
\end{equation}

\medskip \noindent
This means that if we choose our initial $\nu_0=\nu_{U_0}$ by $\nu_0\cdot \lambda={\cal N}_{U_0}(\lambda)$, for all $\lambda \in \Lambda$, the resulting lifting will be periodic.
\end{proof}

\medskip \noindent
In the next section we prove that this type of constrained stress is precisely the periodic stress implicated in the deformation theory of periodic frameworks introduced in \cite{BS2}. 


\section{Periodic deformations and stresses}
\label{sec:deformationPeriodicStress}

We have arrived at one of the most important aspects of this paper, which brings in the connection with the infinitesimal rigidity or flexibility of a periodic framework. The application to periodic pseudo-triangulations and expansive mechanisms presented in the next section relies on this correspondence.

\paragraph{\bf Periodic deformations.} A planar framework $(G,\Gamma,p,\pi)$ was defined by a placement of vertices $p: \V\rightarrow \R^2$ and a faithful representation $\pi: \Gamma \rightarrow {\cal T}(\R^2)$  of the periodicity group by a rank two lattice of translations $\Lambda=\pi(\Gamma)$, with the necessary compatibility relation. In the framework, the edges of the graph are now seen as segments of fixed length, forming what is called in rigidity theory a {\em bar-and-joint structure.} 
According to our periodic deformation theory, introduced in \cite{BS2} and pursued in \cite{BS3,BS4}, a {\em periodic bar-and-joint framework} is said to be {\em periodically flexible} if there exists a continuous (or equivalently smooth) family of placements, parametrized by time $t$, \  $p_t:\V\rightarrow \R^2$ with $p_0=p$, which is not given by global rigid motions and satisfies two conditions: (a) it maintain the lengths of all the edges $e\in \E$, and (b) it maintains periodicity under $\Gamma$, via faithful representations $\pi_t:\Gamma \rightarrow {\cal T}(\R^2)$ which {\em may change with $t$ and give a corresponding variation of the periodicity lattice} $\Lambda_t=\pi_t(\Gamma)$.

\medskip \noindent
To represent $\pi_t$ we first choose two generators for the periodicity lattice $\Gamma$. The corresponding lattice generators $\lambda_1(t)$ and $\lambda_2(t)$ at time $t$ may be viewed as the columns of a non-singular $2\times2$ matrix denoted, for simplicity, with the same symbol $\Lambda_t\in GL(2)$. The infinitesimal deformations of the placement $(p_t,\pi_t)$ are described using a complete set of $n$ vertex representatives for $\V/\Gamma$, i.e. the vertex positions are parametrized by $(\R^2)^n$. The $m$ representatives for  edges mod $\Gamma$ are then expressed using the vertex parameters and the periodicity matrix $\Lambda$. An edge representative $\beta$ originates in one of the chosen vertex representatives $i=i(\beta)$ and ends at some other vertex representative $j=j(\beta)$ {\em plus some period} $\Lambda c_{\beta}$, where $c_{\beta}$ is a column vector with two integer entries. The edge vectors $e_{\beta}$, $\beta\in \E/\Gamma$ thus have the form:

\begin{equation}\label{eq:Edescription}
 e_{\beta}=(x_j+\Lambda c_{\beta})-x_i, \ \ \ \beta\in \E/\Gamma 
\end{equation} 

\medskip \noindent
By taking the squared length of the $m$ edge representatives, we obtain a map:

\begin{equation}\label{map}
 (\R^2)^n\times GL(2) \rightarrow \R^m 
\end{equation}

\medskip \noindent
and the differential of this map at the point under consideration (i.e. the point of $(\R^2)^n\times GL(2)\subset \R^{2n+4}$ corresponding to the framework $(p,\pi)$), seen as a matrix with
$m$ rows and $2n+4$ columns is called the {\em rigidity matrix} ${\bf R}={\bf R}(G,\Gamma,p,\pi)$ of the framework. The row corresponding to the edge described above has the form:

\begin{equation}\label{row}
( 0 ... 0 \ -e_{\beta}^t\  0 ...0\  e_{\beta}^t\  0 ...0 \ c_{\beta}^1 e_{\beta}^t\  c_{\beta}^2  e_{\beta}^t)
\end{equation}

\noindent
where $ e_{\beta}^t$ is the transpose of the column edge vector $e_{\beta}$ and an obvious 
grouping convention is used for the columns.

\medskip \noindent
The vector space of {\em infinitesimal periodic motions} (or {\em infinitesimal periodic deformations}) of the given framework 
$(G,\Gamma,p,\pi)$ can be described as the {\em kernel} of the rigidity matrix ${\bf R}$ and the vector space of
{\em periodic stresses} can be described as the {\em kernel} of the transpose ${\bf R}^t$. It is understood that a stress described on the $m$ representatives for $\E/\Gamma$ is extended by periodicity to all edges.

\medskip \noindent
Thus, non-trivial periodic stresses express linear dependences between the rows of the rigidity matrix ${\bf R}$. Grouping these dependences over groups of columns corresponding to vertex representatives, we obtain immediately the fact that a periodic stress is necessarily a $\Gamma$-invariant equilibrium stress as defined earlier in this paper. However, there are two {\em additional} vector conditions imposed by the columns corresponding to the infinitesimal variation of the periods.

\medskip \noindent
This sets the stage for a comparison of the periodic stresses reviewed here and the 
stresses induced by periodic liftings. It is useful to restate now the definition.

\begin{definition}\label{Pstress}
A {\em periodic stress} for the framework $(G,\Gamma,p,\pi)$ is a 
stress induced from an element in the kernel of the transposed rigidity matrix ${\bf R}^t$, that is, a $\Gamma$-invariant equilibrium stress $s$ satisfying the conditions

\begin{equation}\label{PScond}
 \sum_{\beta\in \E/\Gamma} s_{\beta} c_{\beta}^j e_{\beta}=0\ \ \ j=1,2
\end{equation}

\noindent
with integer coefficients $c_{\beta}^j$ as given in the edge description (\ref{eq:Edescription}).
\end{definition}

\medskip \noindent
{\bf Remarks.} \ The fact that periodic stresses do not depend on the choices of representatives 
used in expressing the rigidity matrix follows from the fact that the image of the map
(\ref{map}) and the image of its differential do not depend on these choices. Periodic stresses
give the orthogonal complement in $\R^m$ for the image of the differential.

There is an  equivalent form for the periodic stress conditions (\ref{PScond}), which refers directly to the
periods $\lambda_{\beta}=\Lambda c_{\beta}$ implicated in the edge descriptions (\ref{eq:Edescription}).
It can be given in terms of tensor products, namely:

\begin{equation}\label{tensor}
 \sum_{\beta\in E/\Gamma} s_{\beta} \lambda_{\beta}\otimes e_{\beta} =0 
\end{equation}

\noindent
The equivalence follows from a straightforward linear algebra verification.

\medskip \noindent
For a comparison of conditions (\ref{PLcond}) and (\ref{PScond}), we verify first the {\em persistence of periodic stresses under relaxation of periodicity} from $\Gamma$ to a subgroup of {\em finite index} $\tilde{\Gamma}\subset \Gamma$.

\begin{prop}\label{prop:persistence}
Let $s=(s_{\beta})_{\beta\in \E}$ be a periodic stress for the periodic framework $(G,\Gamma,p,\pi)$.
Let $\tilde{\Gamma}\subset \Gamma$ be a subgroup of finite index.
Then $s$ remains a periodic stress for the framework with relaxed periodicity $(G,\tilde{\Gamma},p,\pi|_{\tilde{\Gamma}})$.
Moreover, if a $\Gamma$-invariant equilibrium stress becomes periodic for a relaxed periodicity
$\tilde{\Gamma}\subset \Gamma$, it is already periodic for $\Gamma$.

\end{prop}

\medskip \noindent
\begin{proof} A $\Gamma$-invariant equilibrium stress is obviously $\tilde{\Gamma}$-invariant.
We look at the columns related to periods in the two rigidity matrices. 

\medskip \noindent
Let $\rho=card(\Gamma/\tilde{\Gamma})$ denote the index of relaxation of periodicity.
For the relaxed case we use the following representatives. We  first choose representatives in $\Gamma$ for the elements of the quotient group $\Gamma/\tilde{\Gamma}$. Then, a complete set of vertex representatives mod $\tilde{\Gamma}$ will consist of the representatives chosen
mod $\Gamma$ translated by the periods corresponding to the representatives chosen for
$\Gamma/\tilde{\Gamma}$. The subscripts for the relaxed case take the form of pairs
$i,\gamma$, with the $\rho$ representatives $\gamma$ covering $\Gamma/\tilde{\Gamma}$.

\medskip \noindent
The proof amounts to verifying identities which can be better controlled when using a 
tensor product expression for the stress condition on the periodicity columns as in 
\cite{BS2}, page 2641 and as recalled above in (\ref{tensor}). Altogether, with representatives for periodicity $\Gamma$ given in coordinates as

$$ x_i,\ \ i=1,...,n $$

$$ e_{\beta}=x_j+\mu_{\beta}- x_i, \ \ i=i(\beta), j=j(\beta), \mu_{\beta}\in \Lambda $$

\noindent
we have representatives for periodicity $\tilde{\Gamma}$ given in coordinates as

$$ x_{i,\gamma}=x_i+\gamma $$

$$ e_{\beta,\gamma}=(x_j+\gamma)+\mu_{\beta}-(x_i+\gamma)=e_{\beta} $$

\noindent
where, for simplicity, no  further notational distinction has been made for $\gamma\in \Gamma$ and its image by $\pi$. Now, we must take into account that we have unique expressions

\begin{equation}\label{unique}
 \mu_{\beta}=\gamma_{\beta} +\tilde{\lambda}_{\beta} 
\end{equation}

\noindent
with $\gamma_{\beta}$ among the chosen representatives for $\Gamma/\tilde{\Gamma}$
and $\tilde{\gamma}\in \tilde{\Gamma}$. Moreover, as we add representatives, we must record
relations of the form

\begin{equation}\label{hat}
 \gamma + \gamma_{\beta}=\lambda_{\beta,\gamma}+\hat{\lambda}_{\beta,\gamma} 
\end{equation}

\noindent
with $\lambda_{\beta,\gamma}$ among the chosen representatives for $\Gamma/\tilde{\Gamma}$ and $\hat{\lambda}_{\beta,\gamma}\in \tilde{\Gamma}$.

\medskip \noindent
The periodic stress condition for $\Gamma$ reads

\begin{equation}\label{PSfull}
 \sum_{\beta} s_{\beta} \mu_{\beta}\otimes e_{\beta}=0 
\end{equation}

\medskip \noindent
and the one for $\tilde{\Gamma}$ reads

\begin{equation}\label{PSrelax}
 \sum_{\beta,\gamma} s_{\beta} (\tilde{\lambda}_{\beta}+\hat{\lambda}_{\beta,\gamma}) \otimes e_{\beta} =0. 
\end{equation}

\medskip \noindent
In order to see that these conditions are equivalent, we notice that (\ref{unique}) gives

\begin{equation}\label{first}
\sum_{\beta,\gamma} s_{\beta} \tilde{\lambda}_{\beta}\otimes e_{\beta}=
\rho \sum_{\beta} s_{\beta} \mu_{\beta}\otimes e_{\beta}
-\rho \sum_{\beta} s_{\beta} \gamma_{\beta}\otimes e_{\beta} 
\end{equation}

\medskip \noindent
Similarly, from (\ref{hat}), we obtain

\begin{equation}\label{second}
\sum_{\beta,\gamma} s_{\beta} \hat{\lambda}_{\beta,\gamma} \otimes  e_{\beta}=
\sum_{\beta,\gamma} s_{\beta} (\gamma+\gamma_{\beta}-\lambda_{\beta,\gamma})\otimes e_{\beta}.
\end{equation}

\medskip \noindent
Now, we take into account the identity:

\begin{equation}\label{permut}
\sum_{\beta,\gamma} s_{\beta} (\gamma -\lambda_{\beta,\gamma})\otimes e_{\beta} =0
\end{equation}

\noindent
which follows from the fact that when operating with an element of a group, one obtains a permutation of the elements in the group. Here, for any $\beta$, the two lists 
the two lists $\gamma$ and $\lambda_{\beta,\gamma}$ of representatives for $\Gamma/\tilde{\Gamma}$ are 
made of the same elements.

\medskip \noindent
Thus, equations (\ref{first}), (\ref{second}) and (\ref{permut})  imply 

\begin{equation}\label{PSrelaxBis}
\sum_{\beta,\gamma} s_{\beta} (\tilde{\lambda}_{\beta}+\hat{\lambda}_{\beta,\gamma}) \otimes e_{\beta}  =\rho \sum_{\beta} s_{\beta} \mu_{\beta}\otimes e_{\beta}
\end{equation}

\noindent
and establish the equivalence of 
(\ref{PSfull}) and (\ref{PSrelax}). Our proposition is proven.
\end{proof}

\medskip \noindent
{\bf Comment.}\ Tedious as it may be, this verification has the important consequence that
upon relaxation of periodicity, {\em the dimension of the space of periodic stresses can only go up or stay the same.}

\medskip \noindent
We recall (from \cite{BS2}, page 2641) the relation 

\begin{equation}\label{PSm}
\sigma -\delta=m-2n-4
\end{equation}

\noindent
connecting periodic stresses and infinitesimal deformations, where $\sigma$ denotes the dimension of the space of periodic stresses and $\delta$ is the dimension of the space of infinitesimal periodic deformations. Subtracting the trivial infinitesimal deformations induced by infinitesimal isometries, we obtain:
\begin{equation}\label{eq:PSf}
\sigma =\phi -1 +(m-2n)
\end{equation}

\noindent
where $\phi$ denotes the dimension of the space of infinitesimal flexes $\phi=\delta-3$. This formula will be relevant for evaluating behavior under relaxations, with $\sigma$ and  $\phi$ non-decreasing and the term $(m-2n)$ multiplied by the index of relaxation $\rho$.

\medskip \noindent
The last ingredient needed for the proof of our Main Theorem is Lemma \ref{lem:largeCell} below, which can be summarized as saying that {\em for sufficiently relaxed periodicity $\tilde{\Gamma}$, a complete set of edge representatives can be found that are either inside a fundamental parallelogram $P_{\tilde{\Gamma}}$ or cross the border with an endpoint in a neighboring parallelogram, intersecting one of two prescribed spanning sides.} The Lemma gives a more transparent interpretation for the conditions (\ref{PScond}) satisfied by periodic stresses, in terms of a sufficiently large relaxation of periodicity. 

\begin{lem}\label{lem:largeCell}
Let $(G,\Gamma,p,\pi)$ be a planar non-crossing periodic framework.  One can find  generators $\lambda_1,\lambda_2$ for the lattice of periods $\Lambda=\pi(\Gamma)$ and large enough positive integers $r_1$ and $r_2$ such that the relaxation of periodicity given by the sublattice $\tilde{\Gamma}$ of index $\rho=r_1r_2$ of \ \ $\Gamma$ \ corresponding to the two generators $\tilde{\lambda}_j=r_j\lambda_j$, $(j=1,2)$ allows the  following setup:

\medskip
(a)\ a choice of \ fundamental parallelogram $P_{\Gamma}$ spanned by  $\lambda_1, \lambda_2$ such that its boundary avoids vertices and its vertices avoid the framework;

\medskip
(b)\ a representation of the associated fundamental parallelogram $P_{\tilde{\Gamma}}$ for $\tilde{\Gamma}$, spanned
by $\tilde{\lambda}_j=r_j\lambda_j$, $(j=1,2)$ as tiled by $\rho=r_1r_2$ translated copies of
the previous parallelogram;

\medskip
(c)\ using as representatives of vertex-orbits for $\Gamma$ all the vertices inside the fundamental parallelogram $P_{\Gamma}$ and then as representatives of vertex-orbits for $\tilde{\Gamma}$ all the vertices inside the fundamental parallelogram $P_{\tilde{\Gamma}}$, a description of
edge-orbits as used for the two rigidity matrices (for $\Gamma$ and the relaxation $\tilde{\Gamma}$) involving, for $\tilde{\Gamma}$, all edges inside the parallelogram in (b)
and those crossing the two spanning sides given by $\tilde{\lambda}_j, j=1,2$;

\medskip
(d)\ the $\tilde{\Lambda}$ periods implicated in the description of the above crossing edges
are either $-\tilde{\lambda}_1$ or $-\tilde{\lambda}_2$ or possibly their sum 
$-( \tilde{\lambda}_1+\tilde{\lambda}_2)$ .
\end{lem}
 
 \begin{figure}[h]
 \vspace{-10px}
 \centering
  {\includegraphics[width=0.65\textwidth]{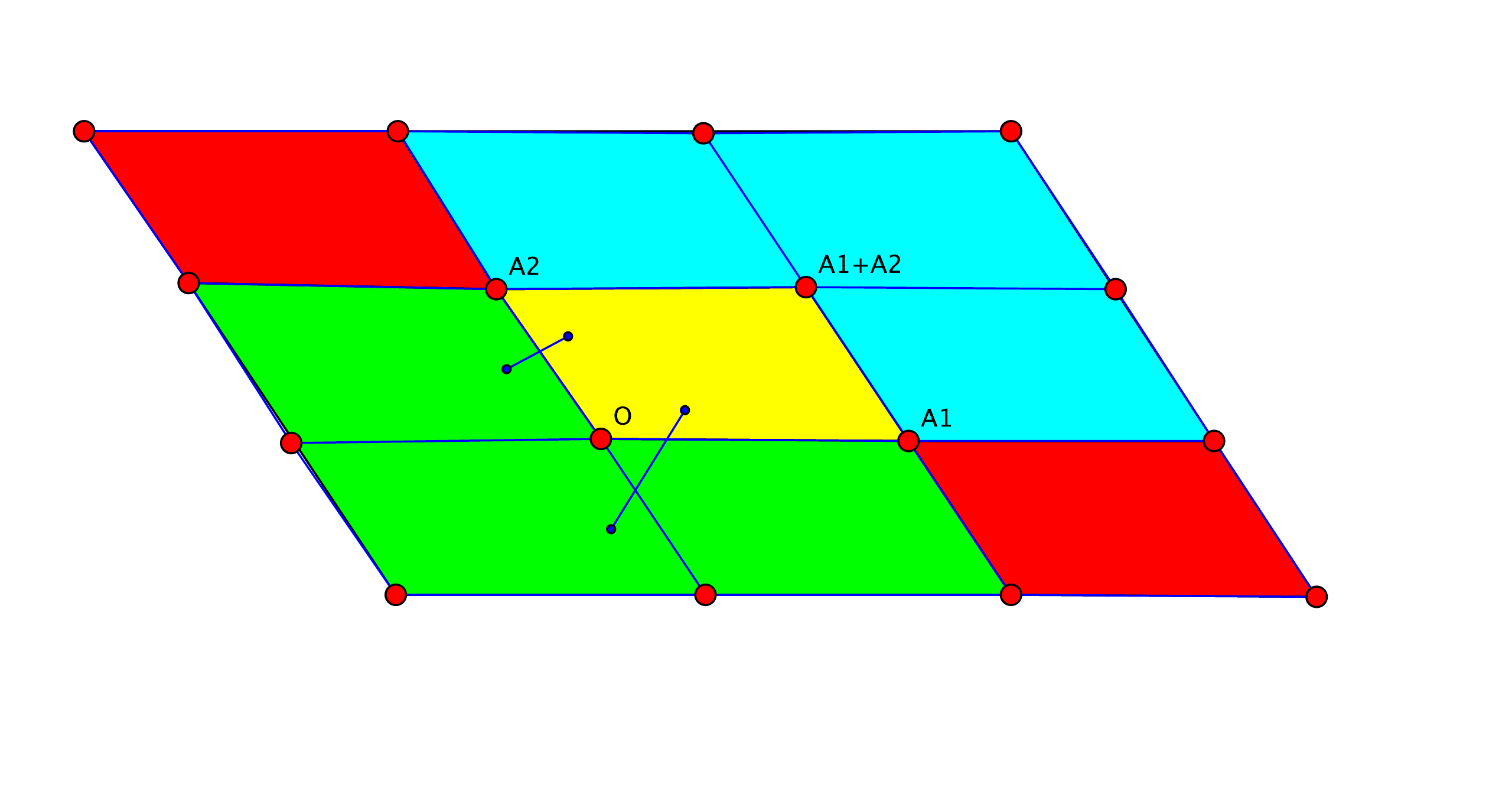}}
 \vspace{-20px}
  \caption{ After a first relaxation one obtains the yellow parallelogram as fundamental domain and
 all edges with one end in it have the other end in it or in one of the  eight neighbours. }
 \label{FigZones}
 \end{figure}
 
\medskip \noindent
{\bf Comment.}\ The intention, content and principle of proof of this lemma can be elucidated by referring to Figure~\ref{FigZones}. 
The aim is to show that convenient choices of  lattice generators and
relaxation can provide a setting where a fundamental parallelogram for the relaxed lattice $\tilde{\Lambda}$ is the central (yellow) zone, marked $OA_1(A_1+A_2)A_2$, with all representatives for edge-orbits contained
in the union of the central zone with the three neighbours (in green) around the O corner. With this corner taken as the origin, the two generators $\tilde{\lambda}_j, j=1,2$
of $\tilde{\Lambda}$ are represented by vectors $ OA_j, j=1,2$, we see that any edge in the union of the 
central zone and the blue zone around $A_1+A_2$ has, by translation with 
$-( \tilde{\lambda}_1+\tilde{\lambda}_2)$, an equivalent representative in the desired union
(of yellow and green). Thus, the lemma can be proven by showing that convenient choices lead to
a setting with no edges crossing from the central (yellow) zone to the remaining two (red) zones abutting corners $A_1$ and $A_2$. In fact, by translation with 
$\tilde{\lambda}_2-\tilde{\lambda}_1$, it is enough to insure that no edge crosses from the central parallelogram to the (red) neighbor abutting at $A_1$.

\medskip \noindent
\begin{proof}\ A generic choice of origin $O$ will easily satisfy the avoidances desired in (a) when we represent the periodicity vectors as emanating from $O$. Assuming such a choice, {\em rational} will refer to the common rational grid determined by the  initial periodicity lattice (and containing all its relaxations).

\begin{figure}[h]
\centering
 {\includegraphics[width=0.60\textwidth]{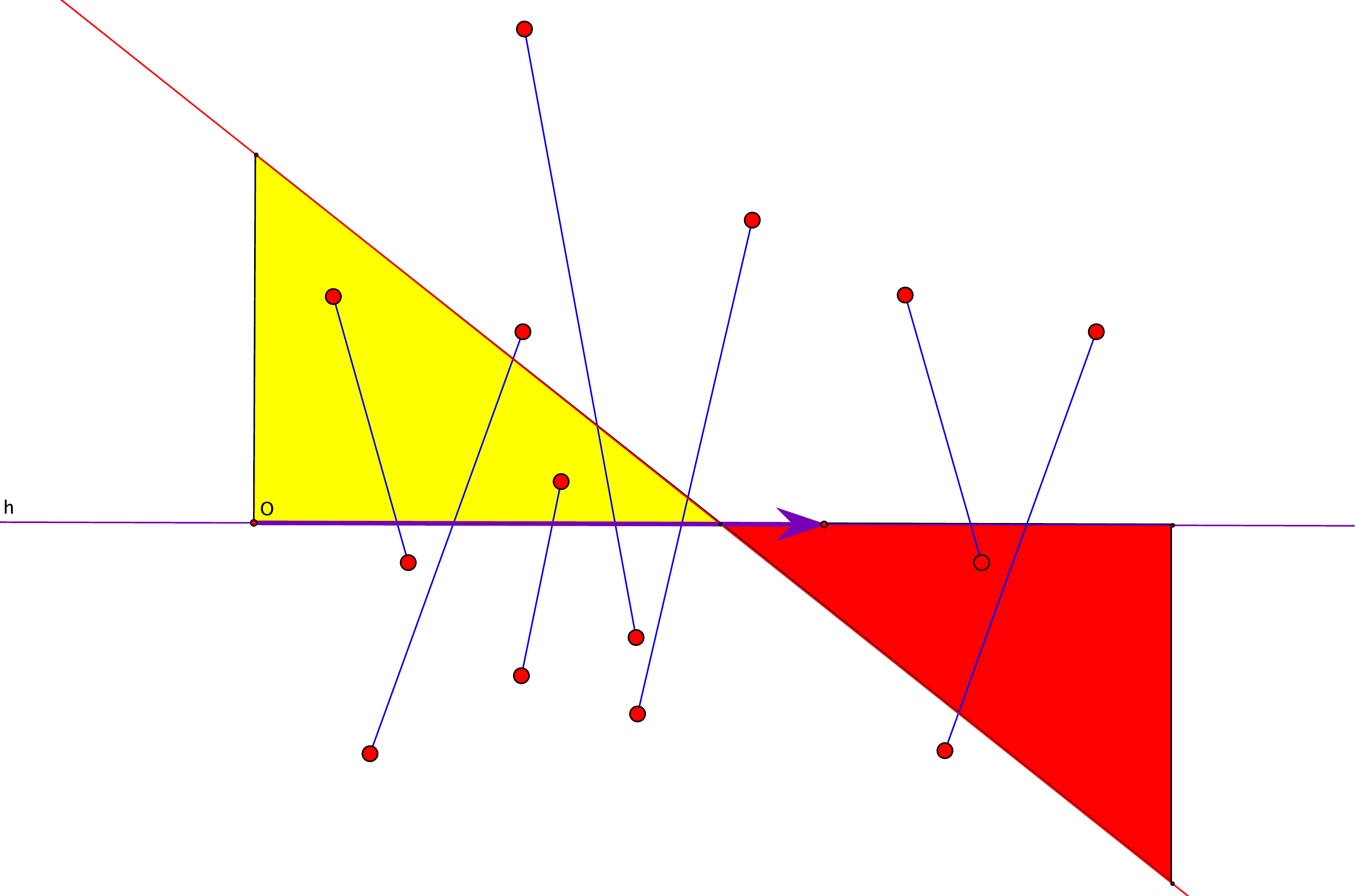}}
 \caption{ The horizontal line $h$ intersected by edges. The marked vector is a period along $h$.
Since all edges have nearly vertical directions a sufficiently more slanted line achieves the property that no edge crosses from the yellow zone to the red zone. }
 \label{FigHcrossing}
\end{figure}

\noindent
Since there is a finite number of edge vectors, their length is bounded and 
it is clear that for any initial choice of fundamental parallelogram $P_{\Gamma}$ we can find
a relaxation with all edges originating in $P_{\tilde{\Gamma}}$ ending either in itself or one of the eight neighboring translated copies. Thus, what remains to be argued is how to obtain the 
aditional property that no edge crosses from the central parallelogram $P_{\tilde{\Gamma}}=OA_1(A_1+A_2)A_2$ to the (red) neighbor abutting at $A_1$.
 
\medskip
For this purpose, we remark that our problem is not affected by a rational linear transformation
and we may use this invariance to arrange and assume that all our edge vectors are, in direction,
sufficiently close to a single direction, which we designate as our ``vertical". We look now at the line through O which is orthogonal to our vertical and refer to is as our ``horizontal" line.

\medskip
Let us depict (with their periodicity along this horizontal) all the edges crossing it, as illustrated
in Figure~\ref{FigHcrossing}.  Then, we can take a more slanted rational line with respect to the vertical as our direction for $\tilde{\lambda}_2$. Under our assumptions, this choice obtains the property that no edge crosses from the yellow region to the red region. We can find first
$\Gamma$-periods along the horizontal and slanted directions. In fact, by the relative freedom 
we have when choosing the slanted direction, we may assume that we obtain generators $\lambda_j, j=1,2$. As already explained, a relaxation $\tilde{\Gamma}$ can be found which satisfies the setting in Figure~\ref{FigZones} and the additional property that all edge representatives for $\tilde{\Gamma}$ can be found in the union of the four parallelograms around O. This completes the proof of our lemma.
\end{proof}

\medskip 
\noindent 
All the elements for the correspondence between periodic liftings and periodic stresses are
now in place to prove:

\medskip
\noindent
{\bf Main Theorem }
{\em
Let $(G,\Gamma,p,\pi)$ be a planar non-crossing periodic framework.
A stress induced by a periodic lifting is a periodic stress and conversely, any periodic stress
is induced by a periodic lifting, determined up to an arbitrary additive constant. The correspondence relates the stress signs to the mountain/valley types of the lifted edges.
}

\begin{proof}\ We use Proposition~\ref{prop:persistence}  and the setting described in Lemma~\ref{lem:largeCell} obtained after an adequate relaxation of periodicity $\tilde{\Gamma}\subset \Gamma$ with generators related by $\tilde{\lambda}_j=r_j\lambda_j, j=1,2$. We first observe that, for periodicity $\tilde{\Gamma}$ the stated correspondence between  periodic liftings and periodic stresses becomes obvious, since conditions (\ref{PScond}) and (\ref{PLcond}) ask exactly the same thing: that the stress-weighted sums of edges implicated along $U\rightarrow U+\tilde{\lambda}_j, j=1,2$ be zero. The case of full periodicity $\Gamma$ now follows from Proposition~\ref{prop:persistence} and the corresponding fact that a $\Gamma$-invariant lifting which is $\tilde{\Gamma}$-periodic for some relaxation $\tilde{\Gamma}\subset \Gamma$, must be already $\Gamma$-periodic,
as immediately seen from conditions (\ref{PLcond}). The sign relationship follows from Proposition~\ref{prop:signStress}.
\end{proof}

\noindent
This concludes the proof. We turn now to applications of our Main Theorem.

\section{Periodic pointed pseudo-triangulations}
\label{sec:periodicPPTs}

In this section we define {\em periodic pointed pseudo-triangulations} or, for short,  {\em periodic pseudo-triangulations}. We prove properties analogous  to those possessed by their finite counterparts \cite{S1,S2}, with respect to  expansive motions. In addition, we show some entirely new properties, specific to the periodic setting, related to {\em ultrarigidity}.

\begin{wrapfigure}{l}{0.35\textwidth}
\vspace{-10pt}
\centering
 {\includegraphics[width=0.35\textwidth]{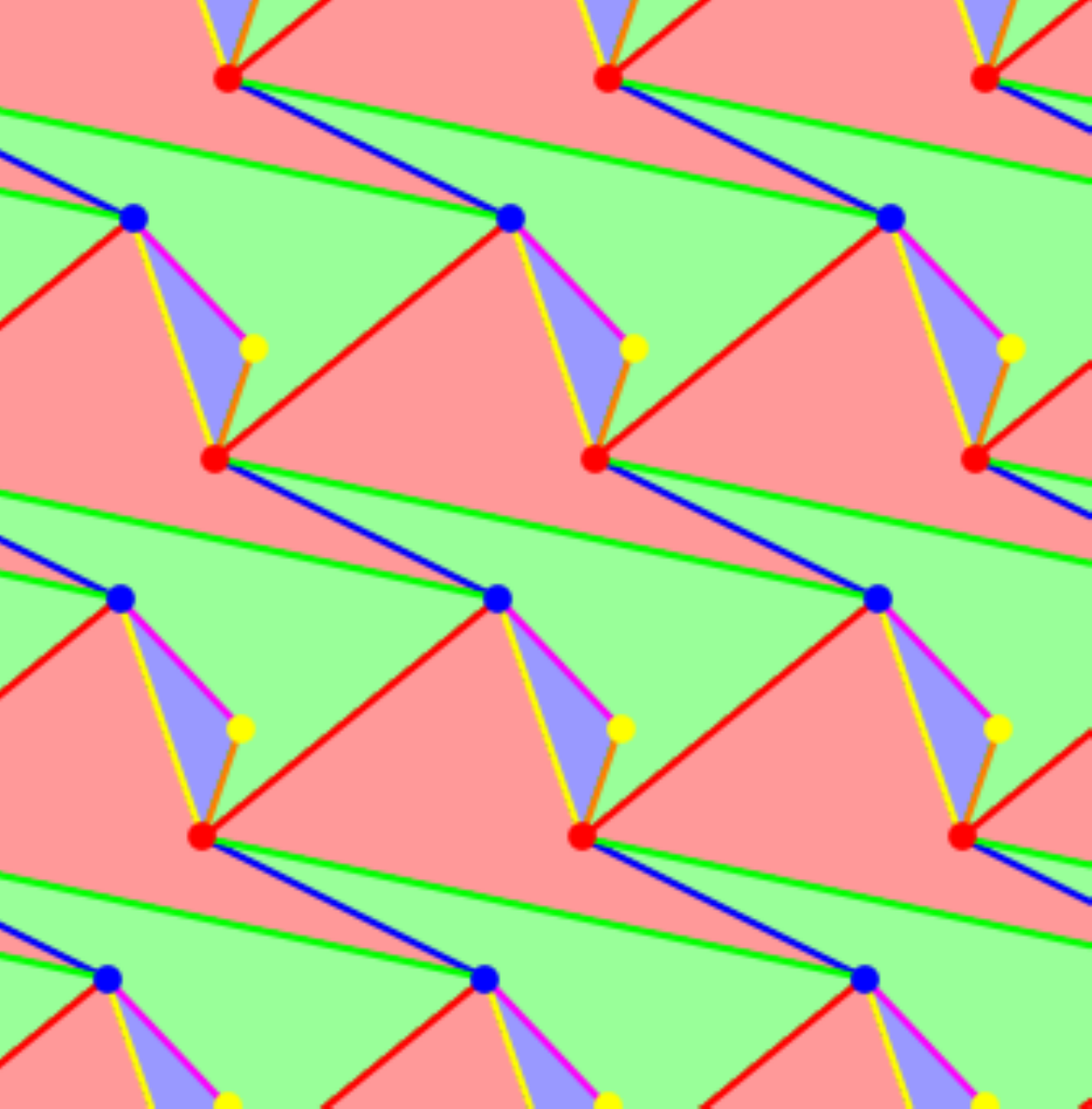}}
  \vspace{-12pt}
 \caption{\small{A periodic pseudo-triangulation with $(n,m,n^*)=(3,6,3)$. }}
   \vspace{-10pt}
 \label{FigPT3}
\end{wrapfigure}

\medskip
  A {\em pseudo-triangle} is a simple closed planar polygon with exactly three internal angles smaller than $\pi$. A set of vectors is {\em pointed} if there is no linear combination with strictly positive coefficients that sums them to $0$. Equivalently, a set of pointed vectors lie in some half-plane determined by a line through their common origin.
A planar non-crossing periodic framework $(G,\Gamma,p,\pi)$ is a {\em periodic pointed pseudo-triangulation} when all faces are pseudo-triangles and the framework is pointed at every vertex. Thus, at every vertex, the incident edges lie on one side of some line passing through the vertex. As in the finite case, pointedness at every vertex is essential. Pseudo-triangular faces mark the `saturated' stage where no more edge orbits can be inserted without violating non-crossing or pointedness. An illustration for $n=3$ is given in Figure~\ref{FigPT3}. We show that periodic pointed pseudo-triangulations, viewed as bar-and-joint mechanisms, satisfy two remarkable rigidity-theoretic properties: they have the right number of edges to be flexible mechanisms with exactly one degree of freedom (in the finite case \cite{S1,S2}, the flexible mechanisms were obtained after removing a convex hull edge), and they encounter no singularities in their deformation for as long as they remain pseudo-triangulations.
  
\begin{prop}\label{prop:mEquals2n}
A periodic pseudo-triangulation has $m=2n$, that is, the number of edge orbits
$m=card(\E/\Gamma)$ is twice the number of vertex orbits $n=card(\V/\Gamma)$.
\end{prop}

\begin{proof}
The proof combines Euler's formula on the torus (\ref{Euler}) with counting the {\em corners} of pseudo-triangular faces. A corner is an angle smaller than $\pi$ at some vertex of a given face.  

\medskip \noindent
We denote by $d_v$ the degree of a vertex (valency) in the quotient multi-graph $G/\Gamma$, and use the classical formula $ \sum_{v\in G/\Gamma} d_v=2m $ relating the degree sum to the number of edges in a graph. Since there are $d_v-1$ corners incident to each vertex, we obtain a second relation:

$$ \sum_{v\in G/\Gamma} d_v=n+3n^*,\ \ \mbox{hence} \ 2m=n+3n^* $$

\medskip \noindent
Combined with Euler's formula (\ref{Euler}) $n-m+n^*=0$ between the number of vertex, edge and face orbits, this yields $m=2n=2n^*$.  An illustration for $n=3$ is given in Figure~\ref{FigPT3}.
\end{proof}

\medskip \noindent
{\bf Comment:} Since for pseudo-triangulations we have $m=2n$, formula (\ref{eq:PSf}) relating the dimensions $\phi$ and $\sigma$ of the spaces of periodic flexes and respectively periodic stresses implies that:
\begin{equation}\label{pstPSf}
\sigma = \phi -1
\end{equation}

\noindent
Moreover, the relation remains unaffected by relaxations of periodicity of finite index. We will make use of this observation further down and in the next section.

\medskip
\noindent
Proposition~\ref{prop:mEquals2n} shows that periodic pointed pseudo-triangulations have the right number of edges to provide smooth one-degree-of-freedom periodic mechanisms. The next proposition shows that this is actually the case. The argument amounts to showing that they have no non-trivial periodic stress. 

\begin{prop}\label{NOstress}
A periodic pseudo-triangulation cannot have nontrivial periodic stresses. The local deformation space is therefore smooth and one-dimensional and continues to be so as long as the deformed framework remains a pseudo-triangulation. The same statement holds true for any relaxation of periodicity $\tilde{\Gamma}\subset \Gamma$ of finite index.
\end{prop}

\begin{proof} A nontrivial stress would give, by our Main Theorem, a periodic lifting.
Such a lifting must have at least one vertex achieving the global maximum and another achieving the global minimum. However, with an angle exceeding $\pi$ on some face around each vertex, neither a maximum nor a minimum is possible.  Thus, the edge length conditions are infinitesimally independent and the implicit function theorem gives a local deformation space
which is smooth and of dimension one. 
\end{proof}

\noindent
We present now a proof of a most remarkable property of periodic pseudo-triangulations.  

\begin{theorem} {\bf (Periodic pointed pseudo-triangulations have expansive 1dof flexes)}\label{thm:expansive}
Let $(G,\Gamma,p,\pi)$ be a planar periodic pseudo-triangulation. Then the framework has a one-parameter periodic deformation, which is expansive for as long as it remains a pseudo-triangulation.
\end{theorem}

\noindent
{\bf Proof outline.} 
We compare the infinitesimal variation of the distance between two pairs of vertices (which do not belong to rigid subcomponents) and show that one cannot increase while the other decreases. An argument already used by Maxwell in \cite{M3} shows that it is enough to reason  with two pairs of vertices which do not create self-crossing when inserted as edges (and replicated by periodicity).  After one edge insertion, the framework becomes infinitesimally rigid and after the second edge insertion it becomes periodically stressed, with a one dimensional periodic stress space. Since either pointedness or non-crossing of edges is violated only by the newly inserted edges, and since a pointed vertex cannot be a local extremum for the height function, we infer that the associated periodic lifting must have its global maximum and global minimum at two of the vertices in the two added edges.  Considerations of stress sign and of the related mountain/valley type of the lifted edges, as summarized in Prop.~\ref{prop:signStress}) show that the new edge orbits have stress factors of opposite signs. This in turn implies (by linear programming duality) that the infinitesimal variation of the distance between the two pairs of vertices is of the same kind: expansion or contraction. 

\medskip
\noindent
{\bf Proof details.} 
We develop our arguments in a sequence of Lemmas. First, we consider what happens after one edge is inserted in a non-crossing manner.

\medskip \noindent
\begin{lem}\label{edge1}
The periodic framework obtained by inserting a new $\Gamma$-orbit of non-crossing edges in a periodic pointed pseudo-triangulation is minimally rigid, that is, infinitesimally rigid with $2n+1$ edge orbits.
\end{lem}

\begin{proof}\ Relation (\ref{eq:PSf}) for the extended framework implies that the dimension of the space of periodic flexes equals the dimension of the space of periodic stresses, i.e. $\sigma=\phi$. We argue now that no nontrivial periodic stress exists, i.e. $\sigma=0$. Indeed, a nontrivial periodic stress would give a nontrivial periodic lifting. The latter must have a global maximum vertex and a global minimum vertex. Since pointed vertices cannot be either maxima or minima,
the two extrema must occur at the endpoints of the new edge. However, a maximum vertex requires at least three non-pointed `mountain' edges and a minimum vertex requires at least three non-pointed `valley' edges. The new edge would have to be a `mountain' because of the maximum at one end and would have to be a `valley' because of the minimum at the other end.
This contradiction proves the lemma.
\end{proof}

\medskip \noindent
We analyze now what happens after a non-crossing second edge orbit insertion. By inserting this new edge and its $\Gamma$-orbit in the minimally rigid framework obtained above, we must have a one-dimensional space of periodic stresses
since the space of infinitesimal flexes remains null. 

\begin{lem}\label{edge2}
If we add two new edge orbits to a periodic pseudo-triangulation, then we cannot have the same sign for the stress factors on the two new edge orbits inserted in the pseudo-triangulation.
\end{lem}

\medskip \noindent
\begin{proof}\ We argue first for the case of no common vertex mod $\Gamma$ for the two new edges. Then, as shown above in Lemma~\ref{edge1}, we cannot have the maximum and the minimum at the two ends of the same new edge. Thus the edge reaching to the maximum is a `mountain' and the other one, reaching to the minimum, is a `valley'.

\medskip \noindent
In case the two new edges share a vertex, we note that at least one extremum must be at
some unshared vertex. But then the other extremum cannot be at the shared vertex and
must be at the other unshared vertex. Thus, as above, the two edges must be a `mountain' and
a `valley'.
\end{proof}

\medskip \noindent
With opposite signs confirmed for the stress factors on the two new edges, we look now 
at the infinitesimal displacement induced by the infinitesimal deformation of the pseudo-triangulation on these same edges. 

\medskip \noindent
For the framework with two new edge orbits, the non-zero stress vector must be orthogonal
on the image of the corresponding rigidity matrix. In particular, it must be orthogonal on the
image vector obtained by evaluating the rigidity matrix on the infinitesimal displacement
induced by the pseudo-triangulation. But all entries corresponding to `old' edge orbits,
that is, edge orbits in the pseudo-triangulation, will be zero. With opposite stress signs
for the two new edges, orthogonality requires the {\em same} sign for the two non-zero
entries. This means: simultaneous infinitesimal expansion or simultaneous infinitesimal
contraction. 

\medskip \noindent
By now, the following result has been established.

\begin{prop}\label{edge12}
Let $(G,\Gamma,p,\pi)$ be a planar pseudo-triangulation. Let two pairs of vertices be such that
both corresponding distances vary infinitesimally under the infinitesimal deformation of the
pseudo-triangulation. Assume that inserting the corresponding two edges (and their $\Gamma$
orbits) does not produce self-crossing. Then, the two edges vary infinitesimally in the same way:
both expand or both contract.
\end{prop}

\begin{wrapfigure}{l}{0.6\textwidth}
\vspace{-24pt}
{\includegraphics[width=0.28\textwidth]{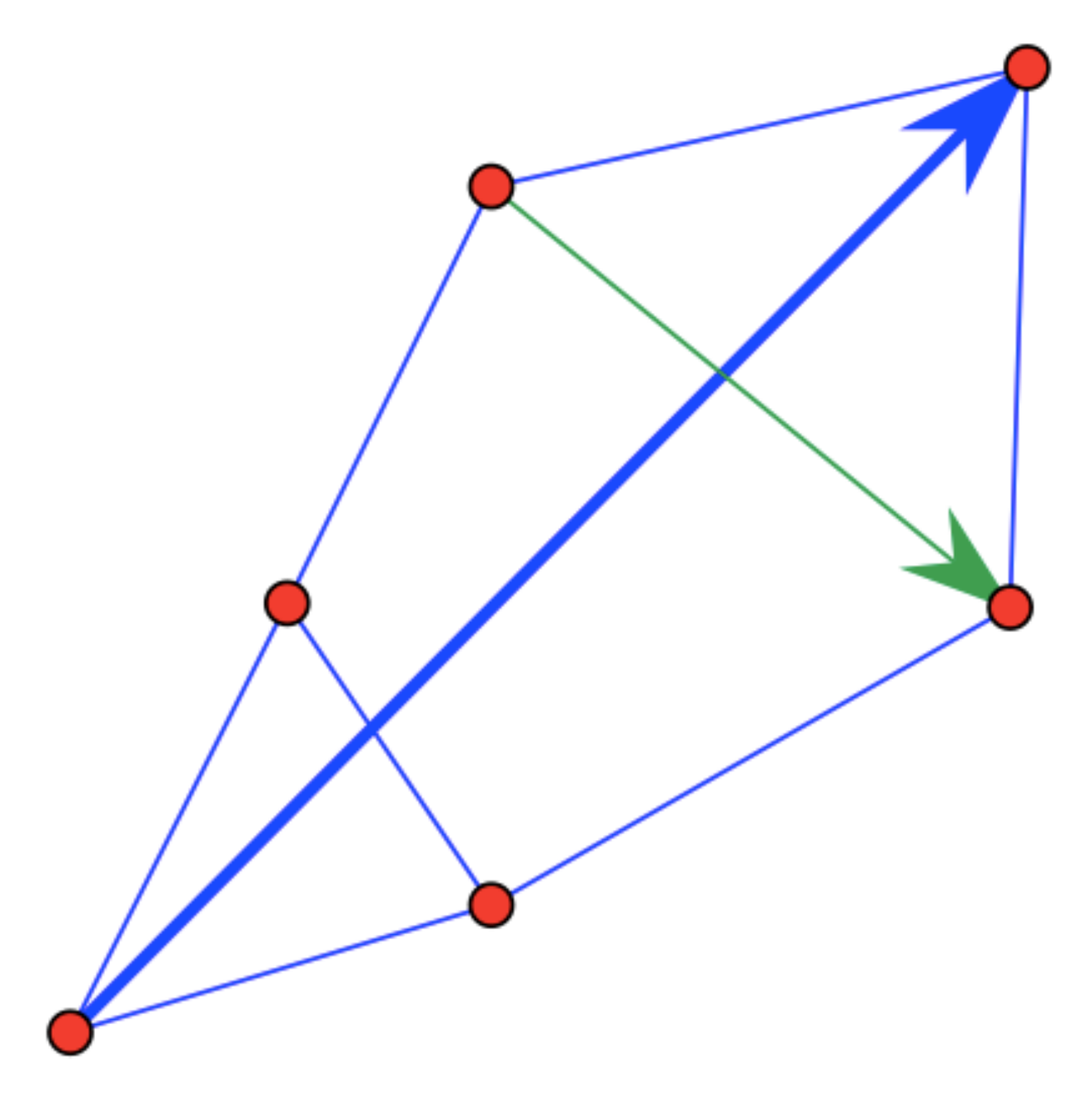}}
\hspace{1pt}
{\includegraphics[width=0.28\textwidth]{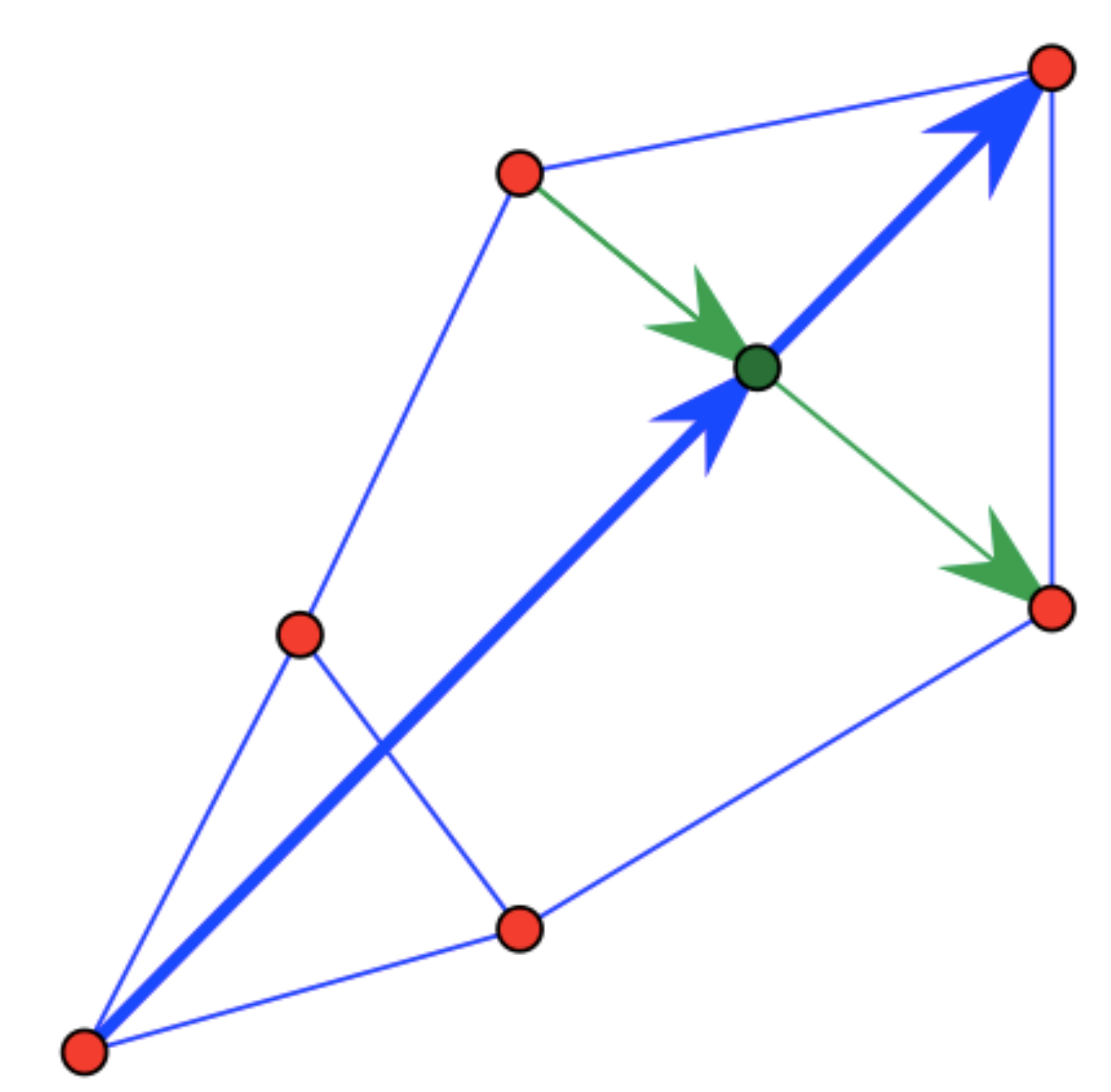}}
 \vspace{-6pt}
 \caption{ Bow's method allows a reduction of stress considerations to a non-crossing case. One step is illustrated above and another one would eliminate self-crossing. }
\vspace{-18pt}
 \label{fig:BowMethod}
\end{wrapfigure}

\medskip \noindent
As a final step to complete the proof of Theorem \ref{thm:expansive}, we want to remove the assumption that the two new edges maintain the non-crossing nature of the periodic framework, and thus to prove the full expansive character of the pseudo-triangulation mechanism. In fact, because we are comparing the stresses and infinitesimal flexes of pairs of edges, it suffices to assume that the first one is not crossing the rest, hence we only need to remove the non-crossing assumption for the second inserted edge.

\medskip \noindent
Figure~\ref{fig:BowMethod} illustrates the (descending) inductive step whereby stress considerations in a situation where an edge intersects several other edges can be reduced to
an equivalent situation without self-crossing. In \cite{M3} the procedure is called Bow's method.

\medskip \noindent
We briefly review the nature of the argument by explaining the stress correspondence for
the step shown in Figure~\ref{fig:BowMethod}. A new vertex has been inserted at the 
last crossing of the marked edge vectors. 
With obvious notations (by which an edge $\alpha$ is split into two edges $\alpha_1$ and $\alpha_2$, and similarly for the edge $\beta$), we have, for the corresponding edge vectors:
$$ e_{\alpha}=e_{\alpha_1}+e_{\alpha_2} 
\ \ \mbox{and} \  \  e_{\beta}=e_{\beta_1}+e_{\beta_2}  $$

\medskip \noindent
Geometric stress conditions at the new vertex require 
$$ s_{\alpha_1}e_{\alpha_1}=s_{\alpha_2}e_{\alpha_2} 
\ \ \mbox{and} \ \  s_{\beta_1}e_{\beta_1}=s_{\beta_2}e_{\beta_2}  $$

\medskip \noindent
and the correspondence of stresses (valid for periodic stresses as well) is expressed through
the formulae:

\begin{equation}\label{Bow}
s_{\alpha}e_{\alpha}=s_{\alpha_1}e_{\alpha_1}=s_{\alpha_2}e_{\alpha_2} \ \ \mbox{and} 
\ \  s_{\beta}e_{\beta}=s_{\beta_1}e_{\beta_1}=s_{\beta_2}e_{\beta_2}
\end{equation}

\noindent
with all other stress factors (on edges different from $\alpha$ and $\beta$) remaining the same.

\medskip \noindent
Relations (\ref{Bow}) show that in this isomorphic correspondence of stresses, the sign
of the stress factors along the fragmented edges is the same as in the initial framework,
hence the argument given in Lemma~\ref{edge2} carries over. Thus, whether crossing or not crossing other edges, the insertion of a second new orbit of edges in a periodic pseudo-triangulation produces a stressed framework with opposite signs for the stress factors on the two added edge orbits. As shown, this implies the same type of infinitesimal variation of the two corresponding  vertex distances when deforming the periodic pseudo-triangulation.  The limit case, when the inserted edge passes through one or several vertices of the framework is treated similarly. In fact, the step for crossing through a vertex is simpler and only requires the
splitting of the inserted edge into two edges.
This concludes the proof of Theorem~\ref{thm:expansive}.
\qed

\medskip \noindent
{\bf Comment.}\ We retain from this section the conspicuous property of periodic pseudo-triangulations of maintaining the same deformation space under arbitrary relaxations of periodicity of finite index. We make use of this property in the final section, where we uncover the ultrarigidity of a family of frameworks obtained from pseudo-triangulations.

\medskip \noindent
We conclude the paper with two applications motivated by questions arising in materials science.

\section{Expansive and auxetic paths}
\label{sec:auxetic}

\noindent
Periodic pseudo-triangulations determine a large, yet distinctive class of planar periodic frameworks. We have presented above some of their remarkable properties: their local deformation space is smooth and one-dimensional and this deformation path is {\em expansive} as long as the framework remains a pseudo-triangulation, that is, for the proper sense of variation of the parameter, all distances between pairs of vertices increase or stay the same. Moreover, a relaxation of periodicity to a subgroup of finite index does not change these characteristics. 

\medskip \noindent
In this section we emphasize, from a {\em purely geometric perspective}, the relevance of 
periodic pseudo-triangulations and their expansive paths for what has been termed 
{\em auxetic behavior} in materials science. 

\medskip \noindent
In materials science, the term {\em auxetic}, suggestive of increase or growth, refers to solids with a negative Poisson's ratio \cite{L,ENHR}.  Quoting from \cite{LWMSE}: {\em``A negative Poisson's ratio in a solid defines the counter-intuitive lateral widening upon application of a longitudinal tensile strain."} For anisotropic solids, such as single crystals, ``the variation of
elastic moduli with direction is also relevant" ({\em ibid.} p.6445). Thus, auxetic behavior is
primarily defined in terms of physical characteristics of the material under consideration.
However, in many instances with pronounced geometric structural underpinnings, geometrical explanations have been proposed  \cite{YHWP,GAE,G-E,MRMST}.

\medskip \noindent
In our context, which is that of periodic frameworks and their deformations, we must rely
on a comprehensive but strictly geometrical concept of auxetic path. This concept 
is introduced in our companion paper \cite{BS5} and will be briefly reviewed here. 

\medskip \noindent
The geometric approach to auxetics presented in \cite{BS5} is formulated in arbitrary dimension
$d$ and addresses one-parameter deformations of a given periodic framework in $\R^d$. Two aspects may  be emphasized from the start: (i) that when present, the auxetic property refers to the deformation path under consideration and not the framework itself, which oftentimes allows other deformation paths which need not be auxetic and (ii) the auxetic character
is an expression of a particular type of variation of the periodicity lattice along the deformation path.

\medskip \noindent
Mathematically, we rely on the notion of {\em contraction operator} in a Hilbert space, which we
recall below. For our purposes, $\R^d$ with the standard norm $|x|=(x\cdot x)^{1/2}$, will sufffice.
Let $T: \R^d \rightarrow \R^d$ be a linear operator. Then, the operator norm, or simply the norm of $T$ is

$$ ||T||=sup_{|x|\leq 1} |Tx|=sup_{|x|=1} |Tx| $$

\noindent
$T$ is called a contraction operator, or simply a contraction, when $||T||\leq 1$.

\medskip \noindent
{\bf Comment.}\ From $|Tx|\leq ||T||\cdot |x|$ it follows that contraction operators are
characterized by the property of taking the unit ball to a subset of itself. For a {\em strict 
contraction} one requires $||T|| < 1$.

\medskip \noindent
Let us assume now that we have a {\em one-parameter} deformation $(G,\Gamma, p_{\tau},\pi_{\tau}), \tau\in (-\epsilon,\epsilon)$ of a periodic framework in $\R^d$. The corresponding periodicity lattices
$\Lambda_{\tau}=\pi_{\tau}(\Gamma)$, offer by themselves a way to compare any two sequential moments $\tau_1 < \tau_2$ by looking at the {\em unique linear operator} $T_{\tau_2\tau_1}$ defined by

\begin{equation}\label{t1t2}
 \pi_{\tau_1}=T_{\tau_2\tau_1}\circ \pi_{\tau_2} 
\end{equation}

\begin{definition}\label{auxetic}
A differentiable one-parameter deformation  $(G,\Gamma, p_{\tau},\pi_{\tau}), \tau\in (-\epsilon,\epsilon)$ of a periodic framework in $\R^d$ is called an {\em auxetic path}, or simply auxetic, when for any $\tau_1 < \tau_2$, the linear operator $T_{\tau_2\tau_1}$ defined by (\ref{t1t2}) is a contraction.
\end{definition}
 
\medskip \noindent
In \cite{BS5} we prove the equivalence of this {\em geometric criterion for auxetic paths}
with the following characterization involving the evolution of the Gram matrix of a generating basis for the period lattice of the framework. With conventions already used in previous sections,
after choosing an independent set of generators for $\Gamma$, the image $\pi_{\tau}(\Gamma)$
is completely described via the $d\times d$ matrix $\Lambda_{\tau}$ with column vectors given
by the images of the generators under $\pi_{\tau}$. The associated Gram matrix will be

$$ \omega(\tau)=\Lambda^t_{\tau}\Lambda_{\tau}.  $$

\begin{prop}\label{tangent}
A deformation path $(G,\Gamma, p_{\tau},\pi_{\tau}), \tau\in (-\epsilon,\epsilon)$ is auxetic if and only if the curve of Gram matrices $\omega(\tau)$ defined above has all its
tangents in the cone of positive semidefinite symmetric $d\times d$ matrices.
\end{prop}

\medskip \noindent
{\bf Remarks.}\ This infinitesimal characterization of auxetic paths is easily seen to be
independent of the choice of generators used for obtaining the Gram matrices. 
There is an obvious analogy here with {\em causal lines} in special relativity. Causal paths must have all their tangents in the {\em light-cone} of Minkowski space-time. Likewise, auxetic paths
must have all their tangents $ \frac{d\omega}{d\tau} (\tau)$ in the cone $\Omega(d)$ of positive semidefinite symmetric matrices. For the geometry and linear symmetries of this cone see \cite{Gr,B}. 

\medskip \noindent
This criterion is particularly convenient when frameworks and deformations are described 
as in \cite{BS4} using parameters in $(\R^d)^{n-1}\times \Omega(d)$. The coordinates in
$(\R^d)^{n-1}$ describe the position of $(n-1)$ chosen representatives for vertex orbits relative to the periodicity basis $\Lambda$ which can be retrieved, up to orthogonal transformations from its Gram matrix $\omega=\Lambda^t\Lambda\in \Omega(d)$. The orbit of a first vertex representative is identified with the periodicity lattice.

\begin{wrapfigure}{l}{0.5\textwidth}
\centering
 {\includegraphics[width=0.24\textwidth]{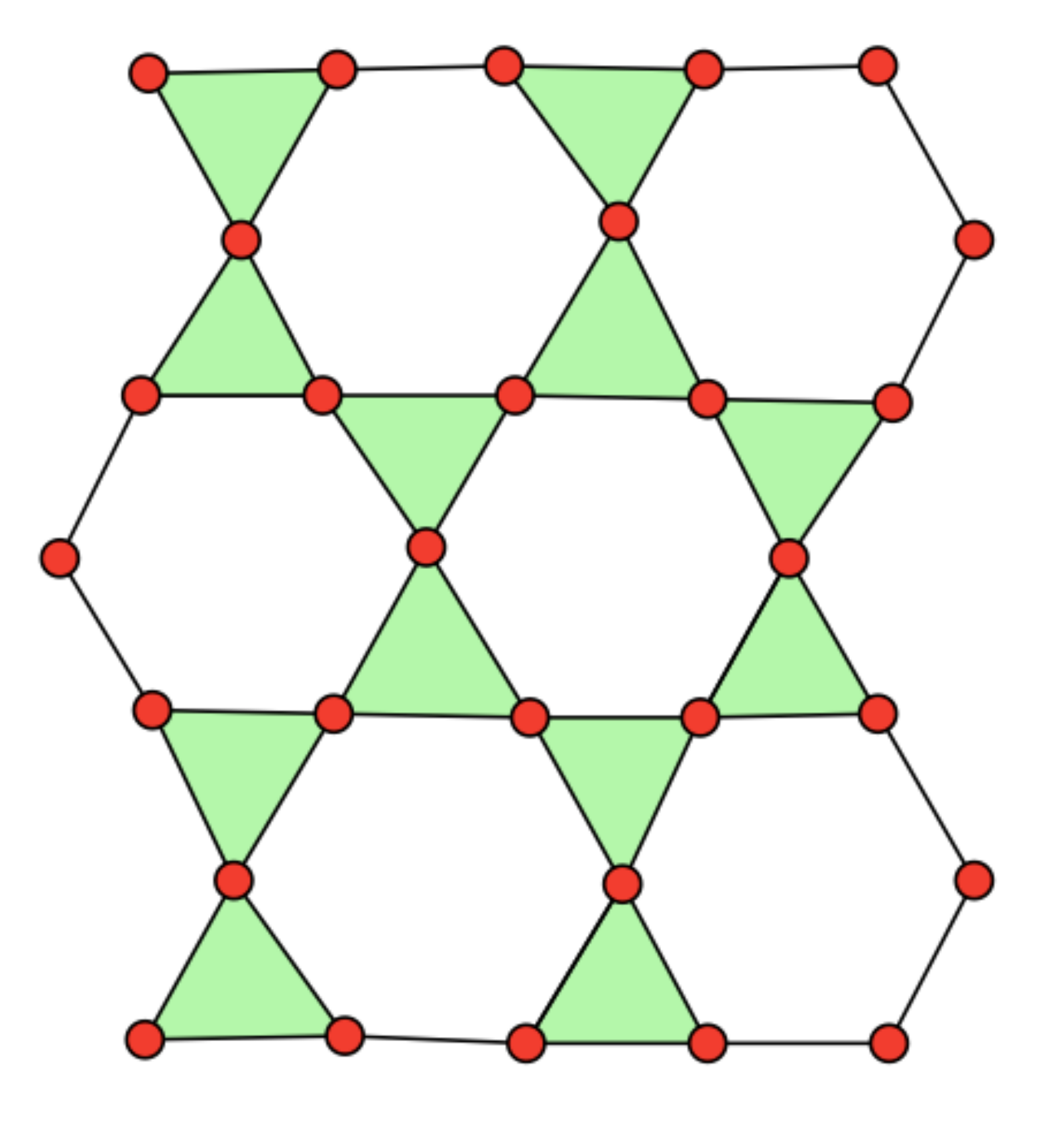}}
 {\includegraphics[width=0.24\textwidth]{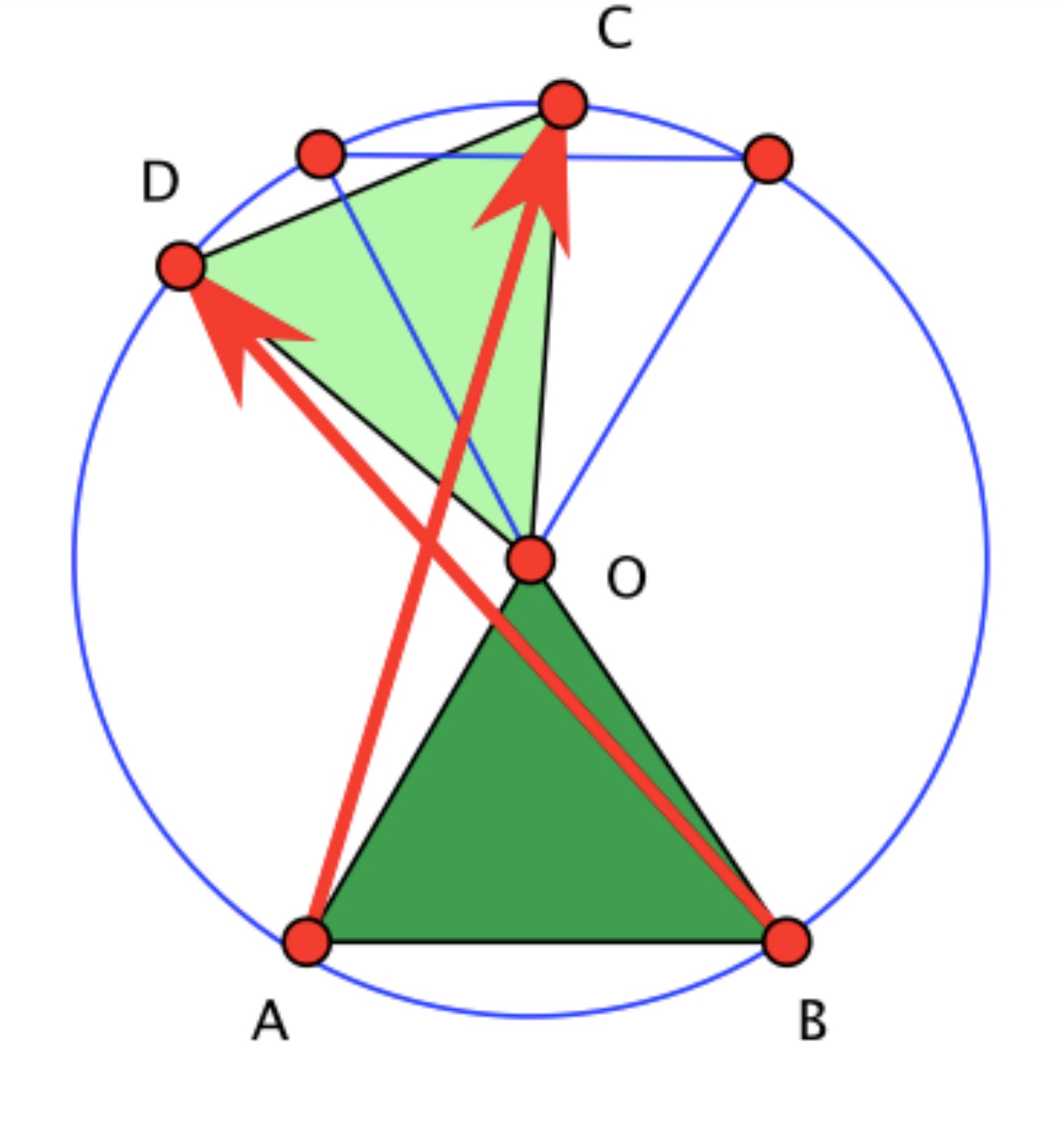}}
\caption{(Left) The Kagome framework. (Right) Parametrizing the deformation of the Kagome framework. Triangle OAB is fixed, and triangle OCD rotates with an angle $\theta$ from the standard position. The two generators of the periodicity lattice are marked as arrows.}
\vspace{-26pt}
\label{fig:kagome}
\end{wrapfigure}

\medskip \noindent
We now address the comparison of expansive and auxetic paths. 

\begin{theorem}\label{expaux}
Let $(G,\Gamma, p_{t},\pi_{t}), t\in (-\epsilon,\epsilon)$ be a one-parameter
deformation of a periodic framework in $\R^d$. If the path is expansive, that is, if the
 distance between any pair of vertices increases or stays the same for increasing $t$, then
the path is also auxetic. However, auxetic paths need not be expansive.
\end{theorem}

\begin{proof} As emphasized earlier, the auxetic property depends only on the curve $\omega(\tau)$ and it will be enough to use the expansive property on one orbit of vertices.
We have to verify that the operator $T_{\tau_2\tau_1}$ which takes the period lattice basis
$\Lambda_{\tau_2}$ to the period lattice basis $\Lambda_{\tau_1}$ is a contraction for $\tau_2 > \tau_1$ 

\medskip \noindent
Let us observe that in the unit ball of $\R^d$, the vectors with rational coordinates relative to the basis $\Lambda_{\tau_2}$ give a dense subset. Since some integer multiple of such a point
is a period at moment $\tau_2$, and this period, as a distance between two vertices in a vertex orbit, can only decrease or preserve its norm when mapped by $T_{\tau_2\tau_1}$ to the corresponding period at moment $\tau_1$, we see that a dense subset of points in the unit ball must be mapped to the unit ball. This is enough to conclude that $||T||\leq 1$.
\end{proof}

\medskip \noindent
The fact that small auxetic deformations need not be expansive is to be expected from the fact that the relative motion of different vertex orbits is not sufficiently constrained by the auxetic property. A simple example is offered by the Kagome framework in dimension two. The Kagome framework is a familiar planar example and has been explored from various points of view \cite{GS,GH,HF,DT,SSML}.  The auxetic character of its one-parameter deformation is frequently mentioned \cite{GAE,M,MRMST}. The brief review  here is meant to offer a simple illustration of our geometric criterion for auxetic paths and to distinguish the expansive portions of this deformation.

\medskip \noindent
\begin{examp}{\bf (Deforming the Kagome framework)}\ 	The basic elements for describing the framework are shown in Figure~\ref{fig:kagome}. The parametrization for the deformation is described in Figure~\ref{fig:kagome}. Triangles $OAB$ and $OCD$ are assumed congruent and equilateral. With origin at $O$, coordinates may be chosen so that $A=(-1,0)$ and $B=-1/2(1,\sqrt 3)$. 
The resulting Gram matrix for the marked periods and a rotation of triangle OCD with angle $\theta$ from the standard position is:

\begin{equation}\label{Gram}
\omega(\theta)=(1+\cos \theta) \left( \begin{array}{cc} 2 & 1 \\
                                                                                        1 & 2   
\end{array}    \right)
\end{equation}

\noindent
with:

\begin{equation}\label{dGram}
\frac{d\omega}{d \theta} (\theta)=-\sin \theta \left( \begin{array}{cc} 2 & 1 \\
                                                                                        1 & 2   
\end{array}    \right)
\end{equation}

\end{examp}

\medskip \noindent
The image of the deformation path in the positive semidefinite cone $\Omega(2)$ for
$\theta \in (-\pi,\pi)$ is the segment from the null matrix to $\omega(0)$, covered twice.
The standard position $\theta=0$ is a maximum for the area of a fundamental parallelogram.
Proposition~\ref{tangent} shows immediately that the symmetric paths obtained for $\theta$
running from $\pi$ to $0$, respectively $-\pi$ to $0$, are auxetic.

\medskip \noindent
The range of $\theta$ corresponding to periodic pseudo-triangulations is the union
$ (-2\pi/3,-\pi/3)\cup (\pi/3,2\pi/3)$. Expansive behavior cannot occur beyond this range, 
since the distance variation between the two pairs of vertices $(A,D)$ and $(B,C)$ has opposite
character when $\theta$ is in the complement: when one segment increases, the other one decreases.

\medskip
\noindent
The framework shown in Fig.~\ref{fig:auxetic2D} is one of the emblematic illustrations of auxetic behavior in dimension two. It is called a {\em reentrant honeycomb} in \cite{PL}. It is often used to illustrate auxetic behavior: a vertical stretch involves a necessary horizontal expansion. 

\medskip \noindent
\begin{examp}{\bf (The ``reentrant honeycomb")}\ 
The framework in Fig.~\ref{fig:auxetic2D} has two degrees of freedom and not all deformations paths are auxetic. The expansive possibilities can be explained in terms of the two possible refinements to periodic pseudo-triangulations shown in Fig.~\ref{fig:pptReentrant}.
\end{examp}

\begin{wrapfigure}{r}{0.54\textwidth}
\vspace{-24pt}
\centering
 {\includegraphics[width=0.26\textwidth]{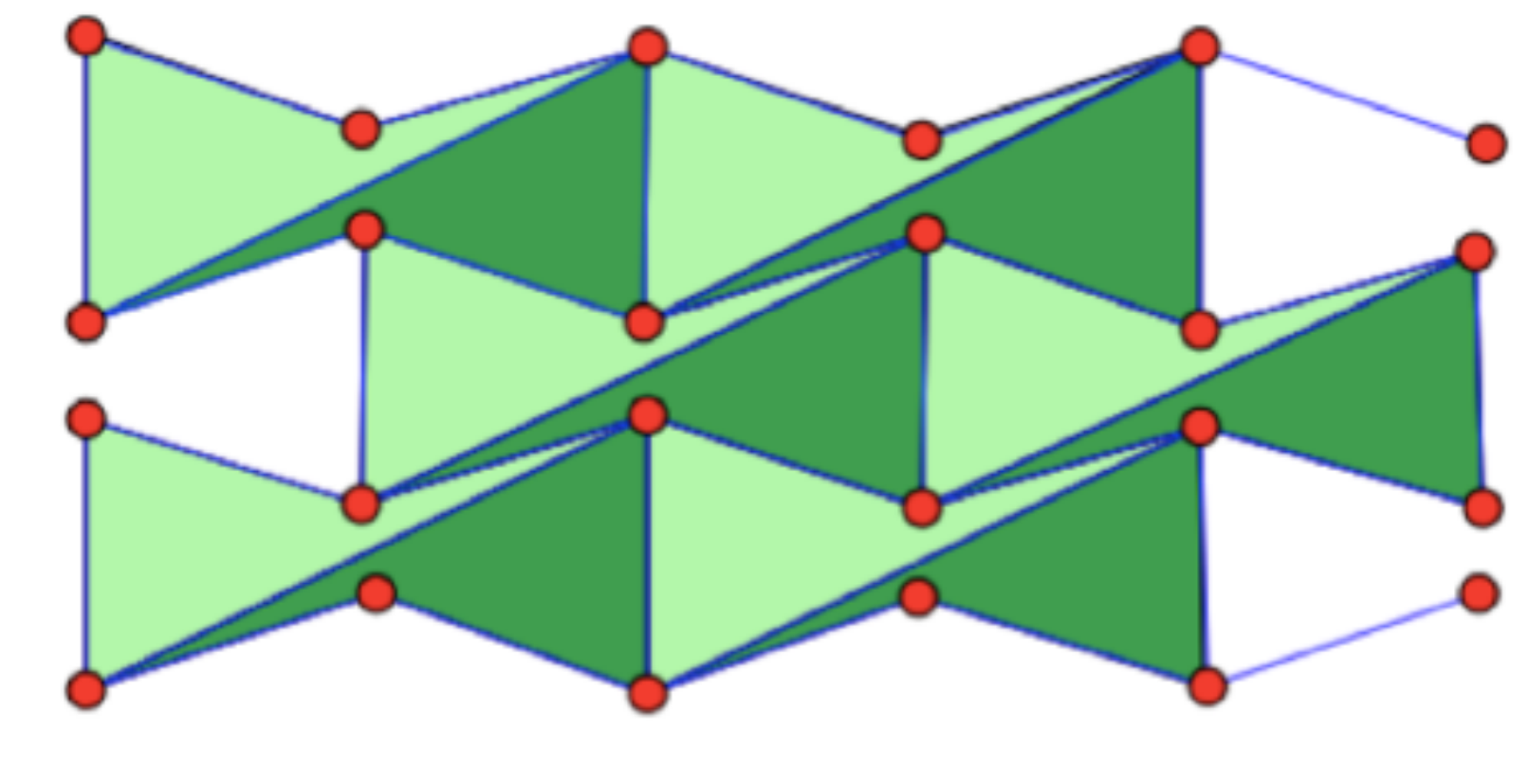}}
 {\includegraphics[width=0.26\textwidth]{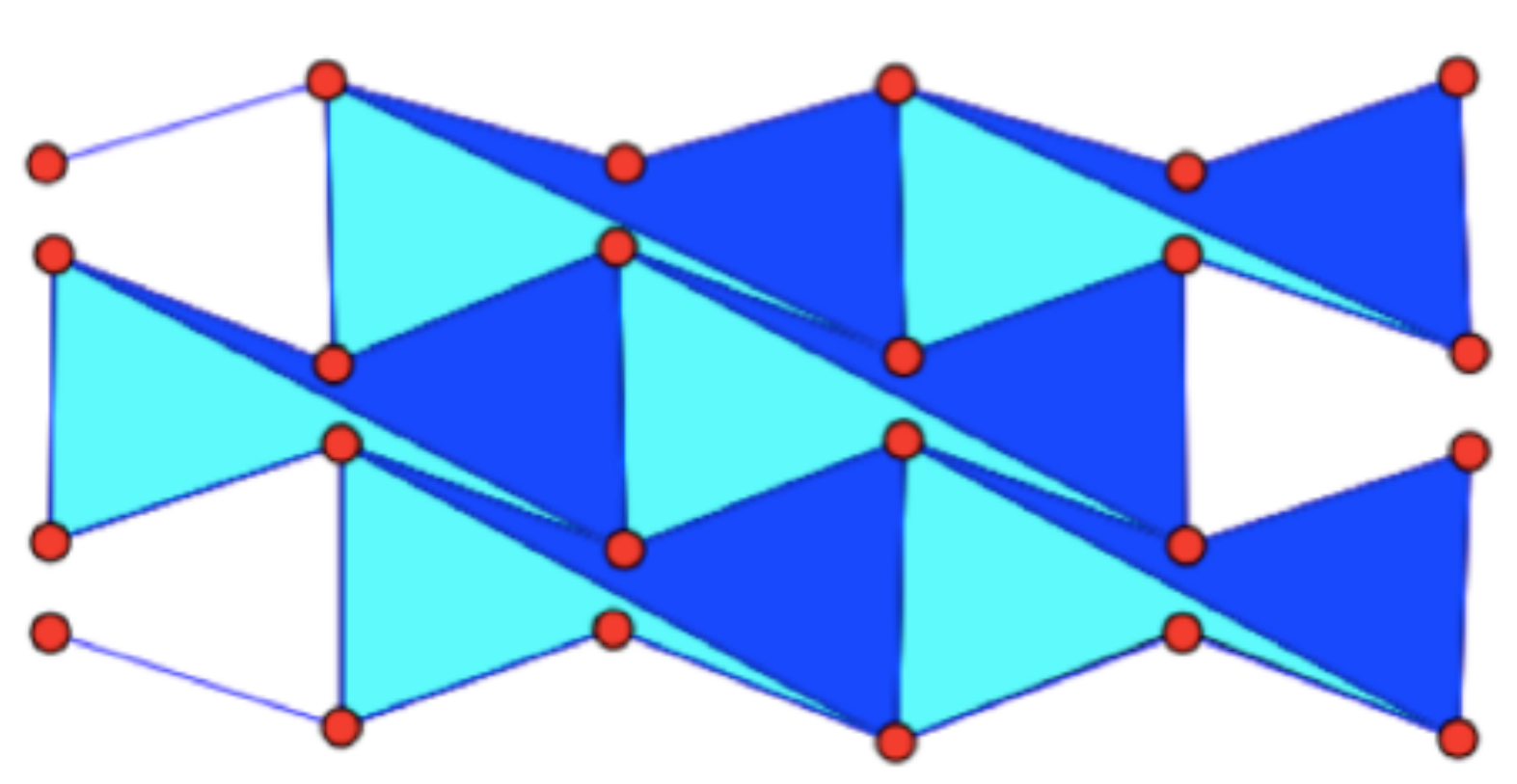}}
 \vspace{-14pt}
 \caption{\small{Two possible refinements to periodic pseudo-triangulations of the `reentrant' structure of hexagons in Figure~\ref{fig:auxetic2D}.}}
\vspace{-14pt}
 \label{fig:pptReentrant}
\end{wrapfigure}

\medskip
\noindent
The framework  with all edges of the same length may be obtained by deforming the 2-diamond framework (the regular hexagonal structure) \cite{BS6}. An auxetic path would result from a `vertical stretch' which visibly entails a horizontal expansion as well. However, an understanding of all expansive infinitesimal deformations within all auxetic possibilities requires some elaboration \cite{BS7}. Here, we limit our discussion to the following remarks on the basic role of periodic pseudo-triangulations.

\begin{wrapfigure}{l}{0.48\textwidth}
\vspace{-12pt}
\centering
 {\includegraphics[width=0.47\textwidth]{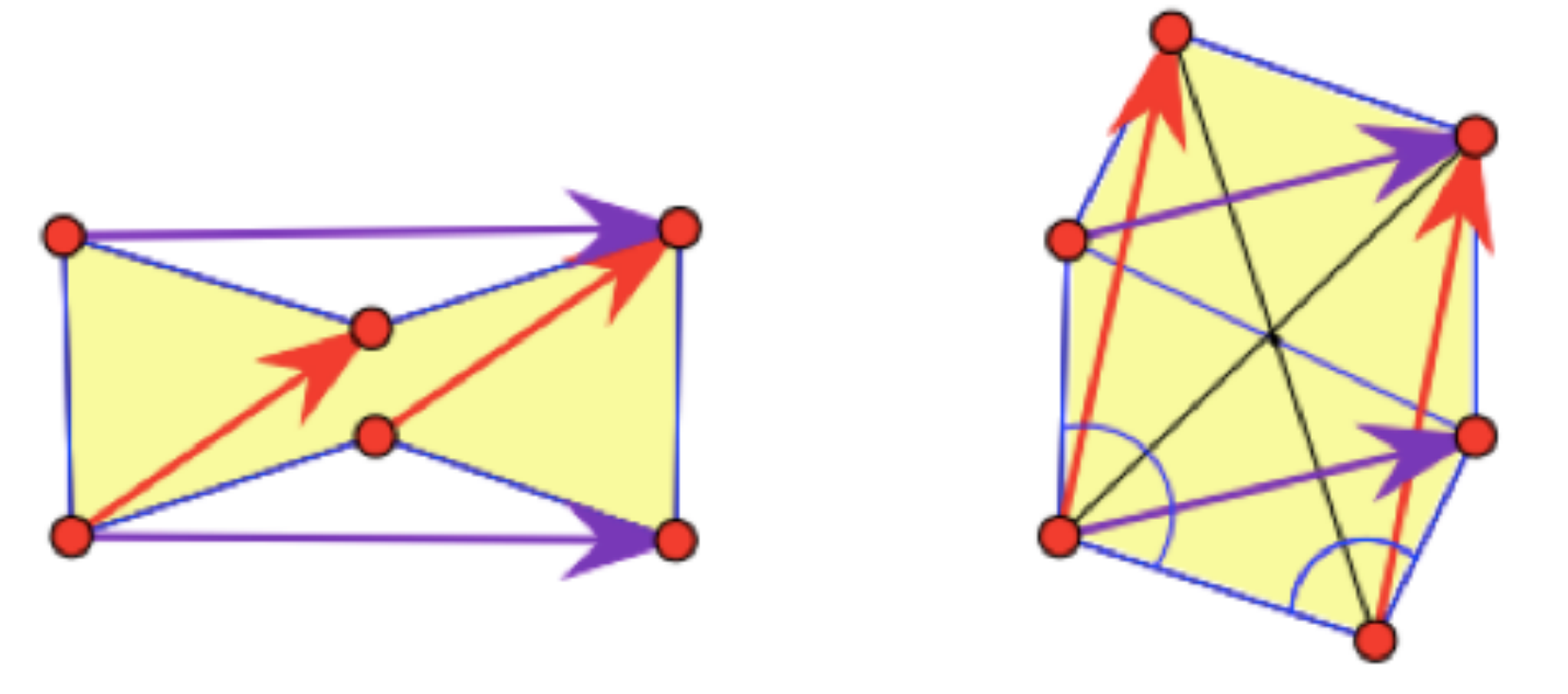}}
 \vspace{-12pt}
 \caption{\small{For the `reentrant' structure of hexagons in Figure~\ref{fig:auxetic2D}, which has two degrees-of-freedom, the figure on the right  marks two angles which may be used as parameters. By maintaining the central symmetry of the hexagon, the two depicted pairs of periods remain equal vectors.}}
\vspace{-10pt}
\label{fig:parametrizationReentrant}
\end{wrapfigure}

\medskip \noindent
The local deformation space of the framework is smooth and two-dimensional. Indeed, by pointedness at every vertex, there can be no periodic stress. Since $n=2$ and $m=3$, there are two degrees of freedom.  Figure~\ref{fig:pptReentrant} marks the two angles which may be used as parameters. By maintaining the central symmetry of the hexagon, the two depicted pairs of periods remain equal vectors. The framework can be refined in two ways to a periodic pseudo-triangulation by insertion of an additional orbit of edges. It can be shown that the expansive infinitesimal deformations determined by these two pseudo-triangulations give the extremal rays of the infinitesimal expansive cone of the framework. For the larger auxetic cone, we refer to \cite{BS5}.

\section{Ultrarigidity}
\label{sec:ultrarigidity}

In this final application we discuss the relevance for ultrarigidity of the other remarkable 
feature of periodic pseudo-triangulations, namely the invariance of their local deformation 
space under arbitrary relaxations of the periodicity group to subgroups of finite index $\tilde{\Gamma}\subset \Gamma$.

\medskip \noindent
The deformation theory of periodic frameworks, as founded and developed in \cite{BS2,BS3,BS4},
is formulated for periodic graphs understood as pairs $(G,\Gamma)$: the periodicity group
$\Gamma$ is an essential structural part of the data. The study of phenomena involving changes
in periodicity or asymptotic behavior under successive relaxations of periodicity may be
intricate. We have recently proposed \cite{borcea:pharmacosiderite:Kavli:arxiv:2012} the notion of {\em ultrarigidity} for expressing the property of a periodic framework $(G.\Gamma,p,\pi)$ of being infinitesimally rigid and remaining so
under any relaxation of periodicity to $\tilde{\Gamma}\subset \Gamma$ of finite index.

\begin{wrapfigure}{r}{0.34\textwidth}
\vspace{-16pt}
\centering
 {\includegraphics[width=0.32\textwidth]{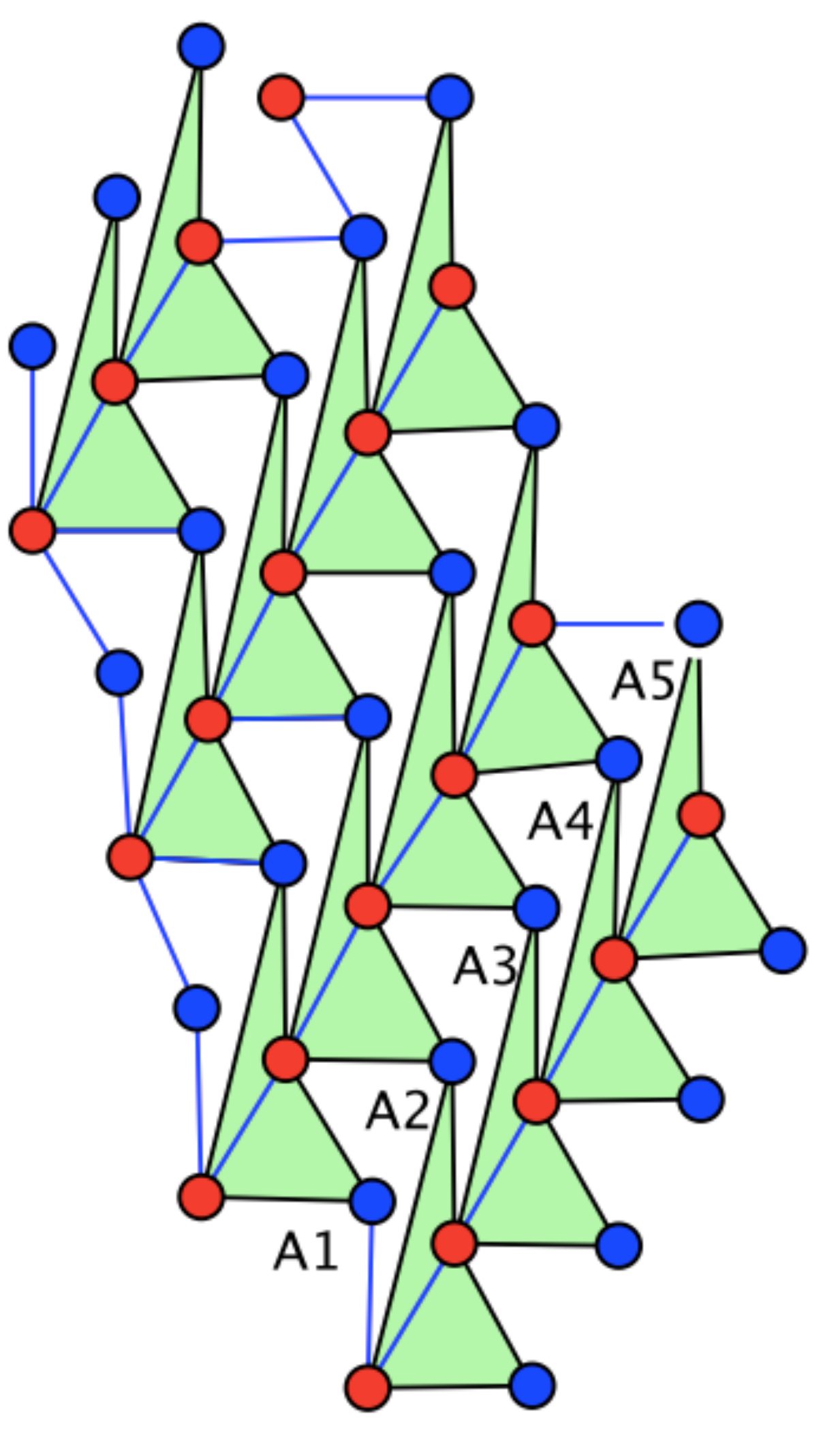}}
\vspace{4pt}
 \caption{ An ultrarigid framework obtained from a periodic pseudo-triangulation with an additional edge orbit. When ignoring periodicity, finite fragments are flexible. 
 }
\vspace{-20pt}
 \label{FigUltrarigid}
\end{wrapfigure}

\medskip \noindent
Ultrarigidity is easily recognized in frameworks made of some rigid finite parts which are
rigidly connected between themselves. However, this is not the general case, as the following
constructions, based on planar periodic pseudo-triangulations, will demonstrate.

\medskip \noindent
Let $(G,\Gamma,p,\pi)$ be a periodic pseudo-triangulation. Then, the local deformation space
is a smooth curve. The infinitesimal deformation corresponding to this one degree of freedom mechanism must induce on some pairs of vertices a non-trivial infinitesimal variation of length.
Then, as argued above in Lemma~\ref{edge1}, by selecting such a pair of vertices and by
inserting the corresponding orbit of edges, we obtain a {\em minimally rigid framework}, that is,
an infinitesimally rigid framework with $m=2n+1$. 

\medskip \noindent
In fact, the resulting framework is {\em ultrarigid}. Indeed, for any relaxation of periodicity
to a subgroup $\tilde{\Gamma}\in \Gamma$ of finite index, the older framework with relaxed
periodicity remains a pseudo-triangulation and has the same local deformation space.
Thus, the  distance between the selected pair of vertices varies infinitesimally and the same
argument applies, showing that insertion of the corresponding $\tilde{\Gamma}$ edge-orbit
already yields an infinitesimally rigid framework. This concludes the proof of the following
result.

\begin{prop}\label{ultrarigid}
Let $(G,\Gamma,p,\pi)$ be a periodic pseudo-triangulation in which we consider a pair of vertices
with a non-trivial infinitesimal variation in distance under the infinitesimal deformation of the
framework. Then the insertion of the corresponding $\Gamma$ orbit of edges results in an
ultrarigid framework.
\end{prop}

\medskip \noindent
Thus, periodic pseudo-triangulations and insertion choices provide endless examples of ultrarigid frameworks.

\medskip \noindent
\begin{examp}{\bf (An ultrarigid periodic framework)}\ The framework illustrated in Figure~\ref{FigUltrarigid} is ultrarigid. The
colors of the vertex orbits indicate the periodicity lattice of the pseudo-triangulation used
in this construction. The pseudo-triangulation itself is a deformed 2-diamond framework with an additional edge orbit \cite{BS6}. The edge orbit which turns the periodic pseudo-triangulation into an ultrarigid framework creates the rigid quadrilaterals
shown in the picture. 

\medskip
\noindent
It is worth remarking that a finite {\em fragment} of the framework may be covered by a fragment resembling the one depicted
in the figure. Such fragments, as {\em finite linkages}, are flexible and  can accommodate  small variations of the segments $A_1A_2$, $A_2,A_3$ etc. However, larger assemblies of these stacked rows  of rigid quadrilaterals will have smaller leeway for variation of deformation parameters. In the limit, as an infinite periodic framework, the structure is rigid.
\end{examp}

\medskip
\noindent
{\bf In conclusion,} we anticipate that our periodic version of Maxwell's Theorem and the expansive nature of periodic pseudo-triangulations will find, like their finite counterparts, further applications in discrete and computational geometry. In the larger scientific context, applications are expected in new materials and mechanism design.

\medskip
\noindent
An extended abstract of this work has appeared in \cite{BS8}.

\vspace{0.5in}

C. Borcea 

\noindent
              Department of Mathematics, Rider University, Lawrenceville, NJ 08648, USA 
              
\medskip
 I. Streinu 

\noindent
              Department of Computer Science, Smith College, Northampton, MA 01063, USA

\end{document}